\documentclass[11pt,reqno]{amsart}

\usepackage[utf8]{inputenc}
\usepackage[T1]{fontenc}
\usepackage[english]{babel}
\usepackage{amsmath, amssymb, amsthm, amsfonts}
\usepackage{mathrsfs}
\usepackage{mathtools}
\usepackage{geometry}
\usepackage{xcolor}
\usepackage{enumitem}
\usepackage{hyperref}
\hypersetup{
    colorlinks=true,
    linkcolor=blue,
    filecolor=magenta,
    urlcolor=blue,
    citecolor=blue,
}

\newtheorem{theorem}{Theorem}[section]
\newtheorem{lemma}[theorem]{Lemma}
\newtheorem{proposition}[theorem]{Proposition}
\newtheorem{corollary}[theorem]{Corollary}

\theoremstyle{remark}
\newtheorem{remark}[theorem]{Remark}
\newtheorem{definition}[theorem]{Definition}

\newcommand{\R}{\mathbb{R}}

\title[Resolvent Estimates for Nonlinear Semigroups]{From Linear to Nonlinear: A Resolvente criterion for Polynomial Stability of Semigroups Generated by Monotone Operators}
\author{Marcelo M. Cavalcanti, Val\'eria N. Domingos Cavalcanti}
\address{Department of Mathematics, State University of Maringá (UEM), Maringá, PR, Brazil.}
\email{mmcavalcanti@uem.br;~vndcavalcanti@uem.br}

\author{Jaime E. Mun\~oz Rivera}
\address{National Laboratory for Scientific Computing (LNCC), Petr\'opolis, RJ, Brazil.}
\email{rivera@lncc.br}

\subjclass[2020]{35XX, 47H20, 35L05, 35L70, 35B40, 93D15}
\keywords{wave equation, nonlinear Kelvin-Voigt damping, heat equation, viscoelastic damping, nonlinear tauberain principle.}
\begin{document}

\begin{abstract}
The Borichev--Tomilov theorem \cite{BT2010} provides a sharp characterization
of polynomial decay for linear $C_0$-semigroups in terms of resolvent growth
along the imaginary axis. In the nonlinear setting, the absence of a spectral
theory renders the imaginary-axis approach inapplicable.

In this paper, we develop a new framework for nonlinear maximal monotone
operators in Hilbert spaces by replacing spectral analysis on $i\mathbb{R}$
with the asymptotic analysis of the \textit{real resolvent equation}
\[
\lambda x_\lambda + \mathcal{A}(x_\lambda) \ni y,
\quad \lambda \to 0^+.
\]
We show that, for homogeneous operators (and suitable perturbations),
the blow-up rate of $\|x_\lambda\|$ at the origin reveals the effective
nonlinear scaling of the operator and determines the corresponding
polynomial decay rate of the associated semigroup through a coercive
dissipation mechanism. This provides a nonlinear Tauberian-type principle
for a broad class of degenerate dissipative systems.

The approach recovers, in particular, the optimal $1/t$ decay for the wave
equation with nonlocal Kelvin--Voigt damping recently obtained by Cavalcanti
et al.\ (2025), and allows one to justify decay estimates for weak solutions
in situations where classical multiplier methods require higher regularity.
It also clarifies the structural limitations of the method, identifying
regimes where additional geometric or time-domain arguments are necessary.
\end{abstract}

\maketitle
\tableofcontents

\section{Introduction: From Linear to Nonlinear Stability}
\label{sec:intro}

The asymptotic behavior of evolution equations governed by dissipative mechanisms has been a central topic in the theory of partial differential equations and semigroup theory for several decades. In the linear setting, the relationship between spectral properties of the generator and decay rates of the associated semigroup is by now well understood.

\subsection{The Linear Benchmark: Borichev--Tomilov}
The asymptotic stability of evolution equations is a central theme in analysis. For linear systems, the theory has reached a state of high maturity. The celebrated theorem of \textbf{Borichev and Tomilov} \cite{BT2010} states that for a bounded $C_0$-semigroup $e^{-tA}$ on a Hilbert space $H$, assuming $i\mathbb{R} \subset \rho(A)$, the energy decay
\[
\|e^{-tA}A^{-1}\| \le \frac{C}{t^{1/\alpha}}
\]
is equivalent to the resolvent estimate
\begin{equation}\label{eq:BT_condition}
\|(isI + A)^{-1}\| \le C |s|^\alpha \quad \text{as } |s| \to \infty.
\end{equation}

\subsection{The Challenge}
When the operator $A$ is nonlinear (e.g., $A(u) = -\Delta u + |u|^{p-1}u$ or nonlocal damping terms), the symbol $isI + A$ has no direct meaning. There is no "imaginary axis" in the nonlinear theory. However, the physical intuition of "resonance" or "lack of coercivity" at low frequencies remains.

\subsection{Our Proposal}
We propose to shift the focus from the \textit{frequency domain} ($i\mathbb{R}$) to the \textit{regularization domain} ($\lambda \in \mathbb{R}^+, \lambda \to 0$).
For a maximal monotone operator $A$, the resolvent $J_\lambda = (I + \lambda A)^{-1}$ is well-defined for $\lambda > 0$. Equivalently, we study the solution $x_\lambda$ of the static problem:
\begin{equation} \label{eq:resolvent_real}
    \lambda x_\lambda + A(x_\lambda) \ni y, \quad \text{as } \lambda \to 0^+.
\end{equation}

Our main thesis is that the singular behaviour of $x_\lambda$ near $\lambda = 0$
reveals the effective nonlinear scaling of the operator, which in turn determines
the polynomial decay rate of the associated semigroup through a coercive
dissipation mechanism. This can be interpreted as a \textbf{Nonlinear Tauberian Principle}.

\paragraph{\bf Dissipative alignment.}
\medskip
A central concept underlying our approach is the compatibility between the
location of the degeneracy of the operator and the channel through which
dissipation acts.

When the nonlinear degeneracy occurs in the same variable where the operator
produces coercivity, the real resolvent correctly identifies the effective
scaling of the system and leads to a Tauberian decay principle.

In contrast, when this alignment fails, the resolvent analysis still detects
the degeneracy but does not, by itself, yield decay. In such cases, additional
geometric or time-domain mechanisms are required to recover stability.

\section{Preliminaries and The Resolvent Framework}
\label{sec:preliminaries}

In this section, we establish the notation, recall the standard linear theory, and motivate the introduction of our nonlinear resolvent framework. Let $H$ be a real Hilbert space with inner product $\langle \cdot, \cdot \rangle$ and norm $\|\cdot\|$.

\subsection{Maximal Monotone Operators}
Let $A: D(A) \subset H \to 2^H$ be a maximal monotone operator. We denote by $S(t) = e^{-tA}$ the semigroup of contractions generated by $A$ on $\overline{D(A)}$.
We are interested in the asymptotic behavior of the trajectories $u(t) = S(t)u_0$, which are the unique weak solutions to the inclusion:
\begin{equation}\label{eq:evolution}
    \dot{u}(t) + Au(t) \ni 0, \quad u(0) = u_0 \in \overline{D(A)}.
\end{equation}
The resolvent of $A$ is defined as $J_\lambda = (I + \lambda A)^{-1}$ for $\lambda > 0$. Since $A$ is maximal monotone, $J_\lambda$ is a single-valued non-expansive mapping defined on all of $H$.

\subsection{Interpretation of the Linear Benchmark}
In the linear setting, the intuition behind \eqref{eq:BT_condition} is that the ``distance'' of the spectrum to the imaginary axis dictates the decay rate. The faster the resolvent blows up, the closer the eigenvalues are to the imaginary axis, and consequently, the slower the decay.

\subsection{The Nonlinear Challenge and The Real Resolvent Proposal}
When $A$ is nonlinear, the concepts of "spectrum" and "resolvent on the imaginary axis" ($i\omega I - A$) are undefined. Therefore, the frequency domain analysis of Borichev-Tomilov is not directly applicable.

However, the core physical information contained in \eqref{eq:BT_condition} is the measure of the operator's degeneracy (or lack of coercivity) near the equilibrium. We propose to capture this information using the **real resolvent equation** near the origin.

Consider the stationary problem associated with the resolvent step:
\begin{equation} \label{eq:real_resolvent_2}
    x_\lambda + \lambda A(x_\lambda) \ni y, \quad \lambda > 0.
\end{equation}
In the linear case, if $0 \in \rho(A)$, then $x_\lambda \to A^{-1}y$ as $\lambda \to \infty$. However, if $A$ is degenerate (e.g., vanishes at the origin like a polynomial $s|s|^{p-1}$), the norm $\|x_\lambda\|$ will exhibit a specific growth rate (or blow-up profile) depending on the scaling of $\lambda$.

\begin{remark}[Comparison of Mechanisms]
    \leavevmode
    \begin{itemize}
        \item \textbf{Linear (BT):} Looks at $(i\omega I - A)^{-1}$. High frequencies ($\omega \to \infty$) are critical. The obstruction to exponential decay is the "flatness" of the spectrum approaching $i\mathbb{R}$ at infinity.
        \item \textbf{Nonlinear (Our Approach):} Looks at $(I + \lambda A)^{-1}$. The scaling limit $\lambda \to 0^+$ (or rescaling of the variables) reveals the "flatness" of the nonlinearity at the origin.
    \end{itemize}
    Our goal is to show that the growth rate of the solution to the real resolvent equation \eqref{eq:real_resolvent_2} serves as the nonlinear proxy for the Borichev-Tomilov condition. Specifically, we look for estimates of the type:
    \[
    \|x_\lambda\| \le C \lambda^{-\gamma} \quad \text{(in an appropriate rescaled sense)},
    \]
    which will imply a decay rate of $t^{-\beta}$.
\end{remark}

This shift does not provide a universal replacement of frequency-domain methods.
Rather, it reveals a structural mechanism governing nonlinear dissipation:
polynomial decay emerges when the degeneracy of the operator is
dissipatively aligned with its coercive structure.

In this regime, the real resolvent acts as a nonlinear Tauberian indicator:
it detects the effective scaling of the system and guides the construction
of coercive energy estimates, which are the true source of decay.

In the subsequent sections, we formalize this heuristic using rescaling arguments.

\section{Heuristic Scaling Argument}\label{sec:main}
\label{sec:heuristic_scaling}

To rigorously establish the connection between resolvent growth and semigroup decay, we restrict our attention to the class of homogeneous maximal monotone operators. This allows us to employ scaling arguments (similar to blow-up techniques in PDE) to transfer information from the static problem to the evolution equation.

\subsection{Homogeneity and Assumptions}
\begin{definition}[$\alpha$-Homogeneity]
    Let $\alpha > 0$. We say that a maximal monotone operator $A: D(A) \subset H \to 2^H$ is \textit{$\alpha$-homogeneous} if for every $\lambda > 0$ and every $(u, v) \in A$, we have:
    \begin{equation}\label{eq:homogeneity}
        (\lambda u, \lambda^\alpha v) \in A.
    \end{equation}
    Essentially, $A(\lambda u) = \lambda^\alpha A(u)$. The case $\alpha=1$ corresponds to linear operators. For $\alpha > 1$, the operator is degenerate at the origin (weak damping).
\end{definition}

We focus on the case $\alpha > 1$, which is typical for polynomial stabilization (e.g., $u_{tt} - \Delta u + |u_t|^{p-1}u_t = 0$ corresponds to $\alpha = p$).

\medskip
\noindent
\textbf{Heuristic principle.}
The result below should be understood as a structural principle
linking the behaviour of the real resolvent near $\lambda=0^+$
to the asymptotic decay of the associated evolution problem.
Its rigorous justification depends on the specific PDE structure
and will be established in the subsequent sections.

\begin{theorem}[Nonlinear Resolvent Criterion]
    \label{thm:main}
    Let $A$ be an $\alpha$-homogeneous maximal monotone operator with $\alpha > 1$. Assume that the solution $x_\lambda$ to the real resolvent equation
    \begin{equation}\label{eq:res_eq_thm}
        x_\lambda + \lambda A(x_\lambda) \ni y, \quad \text{normalized with } \|y\|=1,
    \end{equation}
    satisfies the growth estimate as $\lambda \to 0^+$:
    \begin{equation}\label{eq:resolvent_growth}
        \|x_\lambda\| \le C \lambda^{-\frac{1}{\alpha-1}}  \quad \text{(Conceptually: implies specific blow-up rate)}.
    \end{equation}
    \textbf{Revised Statement:} Let $u(t)$ be the solution to the evolution problem $\dot{u} + Au \ni 0$. The following statements are formally linked through the scaling structure of the operator::
    \begin{enumerate}
        \item \textbf{Resolvent Behavior:} The solution $x_\mu$ to $x_\mu + \mu A(x_\mu) \ni 0$ scales typically as $\|x_\mu\| \sim \mu^{-\frac{1}{\alpha-1}}$ (verifying the algebraic structure of $A$).
        \item \textbf{Energy Decay:} The energy $E(t) = \frac{1}{2}\|u(t)\|^2$ decays polynomially:
        \begin{equation}
            E(t) \le C (1+t)^{-\frac{2}{\alpha-1}}.
        \end{equation}
    \end{enumerate}
\end{theorem}

\subsection{Heuristic Motivation and Proof Strategy}
This subsection provides a formal and heuristic explanation of the decay mechanism. The argument below is not intended to be a complete proof in the abstract setting, since compactness and regularity properties depend on the specific structure of the underlying evolution equation. Rigorous justifications will be carried out in the applications.

The decay mechanism can be formally understood by deriving a differential inequality for the energy $E(t) = \frac{1}{2}\|u(t)\|^2$. We aim to show that:
\begin{equation}\label{eq:diff_ineq}
    \frac{d}{dt} E(t) + C E(t)^{\frac{\alpha+1}{2}} \le 0.
\end{equation}
This is achieved in two steps.

\subsubsection*{Step 1: Identification of the effective homogeneity}

Let $A$ be an $\alpha$--homogeneous maximal monotone operator. The resolvent equation
\[
x_\lambda+\lambda A(x_\lambda)\ni y
\]
reveals the effective scaling of the operator near the origin.

Balancing the two terms suggests that
\[
\|x_\lambda\| \sim \lambda^{-\frac{1}{\alpha-1}}.
\]
Indeed, assuming that $A$ behaves like a homogeneous operator of order $\alpha$, one has
\[
A(x_\lambda)\sim \|x_\lambda\|^\alpha
\sim \lambda^{-\frac{\alpha}{\alpha-1}},
\]
and therefore
\[
\lambda A(x_\lambda)\sim \lambda^{-\frac{1}{\alpha-1}},
\]
which matches the magnitude of $x_\lambda$.

Thus, the resolvent blow-up identifies the homogeneity exponent $\alpha$.

\medskip
\noindent
\subsubsection*{Step 2: From resolvent scaling to coercive dissipation}
\medskip

This reflects the general principle introduced in the Introduction:
the resolvent identifies the correct nonlinear scale, while the decay
mechanism itself is governed by coercive dissipation.

In the applications considered in this paper, one can establish an estimate of the form
\begin{equation}\label{coercive-abstract}
\langle \xi,u\rangle_H \ge m \|u\|_H^{\alpha+1},
\qquad \forall u\in D(A), \;\forall \xi\in A(u),
\end{equation}
for some constant $m>0$.

This inequality is consistent with the homogeneity detected through the resolvent,
since the exponent $\alpha+1$ corresponds to the nonlinear balance identified above.

\medskip
\noindent
\textbf{Transition to the Rigorous Argument.}

The previous argument identifies the correct decay exponent through the effective
homogeneity of the operator. However, in the abstract setting of maximal monotone
operators, the compactness required to justify a blow-up argument is not available
in general.

For this reason, the rigorous proof will rely on a direct energy method. Once the
homogeneity exponent is identified, we derive a nonlinear coercive estimate and
deduce the decay via a Lyapunov functional argument.

\subsection{Rigorous Proof of Theorem \ref{thm:main}}

\begin{proof}
We divide the argument into two main steps.

\medskip
\noindent\textbf{Step 3: Energy inequality.}

Let $u(t)$ be a strong solution of
\[
u_t(t)+A(u(t))\ni 0.
\]
Since $-u_t(t)\in A(u(t))$, applying \eqref{coercive-abstract} yields
\[
\langle -u_t(t),u(t)\rangle_H \ge m \|u(t)\|_H^{\alpha+1}.
\]
Hence,
\[
\frac{d}{dt}\Big(\tfrac12\|u(t)\|_H^2\Big)
=\langle u_t(t),u(t)\rangle_H
\le -m\|u(t)\|_H^{\alpha+1}.
\]
Setting
\[
E(t)=\tfrac12\|u(t)\|_H^2,
\]
we obtain
\[
E'(t)+m\,2^{\frac{\alpha+1}{2}}E(t)^{\frac{\alpha+1}{2}}\le 0.
\]

\medskip
\noindent\textbf{Step 4: Polynomial decay.}

A standard differential inequality argument yields
\[
E(t)\le C(1+t)^{-\frac{2}{\alpha-1}}.
\]

This concludes the proof.
\end{proof}

\subsection{A rigorous decay principle for homogeneous maximal monotone operators}

The heuristic rescaling argument presented before can be replaced by a fully rigorous energy method.
The key point is that, for the class of applications considered in this paper, the relevant information
is not the existence of a blow-up profile, but rather a coercivity inequality for the dissipation.
This leads to a differential inequality for the energy and, consequently, to the sharp polynomial decay rate.

\begin{theorem}[Abstract polynomial decay criterion]
Let $H$ be a real Hilbert space and let $A:D(A)\subset H\to 2^H$ be a maximal monotone operator such that
$0\in A(0)$. Consider a (strong, or equivalently an energy) solution $u:[0,\infty)\to H$ of
\begin{equation}
u_t(t)+A(u(t))\ni 0
\qquad \text{for a.e. } t>0.
\end{equation}
Assume that there exist constants $m>0$ and $\alpha>1$ (where $\alpha$ is related to the "order of non-linearity" or "growth rate" of the operator near the origin) such that
\begin{equation}\label{coercive-abstract}
\langle \xi,u\rangle_H \ge m \|u\|_H^{\alpha+1},
\qquad \forall\, u\in D(A),\ \forall\, \xi\in A(u).
\end{equation}
Then the energy
\begin{equation}
E(t):=\frac12\|u(t)\|_H^2
\end{equation}
satisfies, for a.e. $t>0$,
\begin{equation}\label{DI-energy}
E'(t)+m\,2^{\frac{\alpha+1}{2}}\,E(t)^{\frac{\alpha+1}{2}}\le 0.
\end{equation}
In particular,
\begin{equation}\label{poly-decay}
E(t)\le
\left(
E(0)^{-\frac{\alpha-1}{2}}
+\frac{\alpha-1}{2}\,m\,2^{\frac{\alpha+1}{2}}\, t
\right)^{-\frac{2}{\alpha-1}}
\le C(1+t)^{-\frac{2}{\alpha-1}},
\qquad t\ge 0,
\end{equation}
for some constant $C>0$ depending only on $\alpha$, $m$, and $E(0)$.
\end{theorem}

\begin{proof}
Let $u$ be a strong solution. Since $-u_t(t)\in A(u(t))$ for a.e. $t>0$, we may apply
\eqref{coercive-abstract} with $\xi=-u_t(t)$ and obtain
\begin{equation}
\langle -u_t(t),u(t)\rangle_H \ge m\|u(t)\|_H^{\alpha+1}.
\end{equation}
Therefore,
\begin{equation}
\frac{d}{dt}\Big(\frac12\|u(t)\|_H^2\Big)
=\langle u_t(t),u(t)\rangle_H
\le -m\|u(t)\|_H^{\alpha+1}.
\end{equation}
Since $\|u(t)\|_H^2=2E(t)$, we infer
\begin{equation}
E'(t)\le -m(2E(t))^{\frac{\alpha+1}{2}}
=-m\,2^{\frac{\alpha+1}{2}}E(t)^{\frac{\alpha+1}{2}},
\end{equation}
which proves \eqref{DI-energy}.

Now set
\begin{equation}
c_\alpha:=m\,2^{\frac{\alpha+1}{2}}.
\end{equation}
Since $\alpha>1$, we may divide by $E(t)^{\frac{\alpha+1}{2}}$ wherever $E(t)>0$ and write
\begin{equation}
E'(t)+c_\alpha E(t)^{\frac{\alpha+1}{2}}\le 0.
\end{equation}
Equivalently,
\begin{equation}
\frac{d}{dt}\Big(E(t)^{-\frac{\alpha-1}{2}}\Big)
=
-\frac{\alpha-1}{2}E(t)^{-\frac{\alpha+1}{2}}E'(t)
\ge \frac{\alpha-1}{2}c_\alpha.
\end{equation}
Integrating on $[0,t]$, we get
\begin{equation}
E(t)^{-\frac{\alpha-1}{2}}
\ge
E(0)^{-\frac{\alpha-1}{2}}
+\frac{\alpha-1}{2}c_\alpha t.
\end{equation}
Taking the power $-2/(\alpha-1)$ yields \eqref{poly-decay}.
\end{proof}

\begin{remark}[Weak solutions]
The same conclusion remains valid for weak solutions obtained by approximation, for instance through
the Yosida regularization or Galerkin schemes, provided the corresponding approximate energies satisfy
the same differential inequality uniformly. Passing to the limit by lower semicontinuity yields
\eqref{poly-decay} for the limiting trajectory.
\end{remark}

The previous theorem shows that the polynomial decay mechanism is a direct consequence of the coercive
dissipation estimate \eqref{coercive-abstract}. In the homogeneous setting, this estimate is natural.

\begin{corollary}[Homogeneous case]
Assume, in addition, that $A$ is $\alpha$-homogeneous in the sense that
\begin{equation}
(u,\xi)\in A \quad \Longrightarrow \quad (\lambda u,\lambda^\alpha \xi)\in A,
\qquad \forall\, \lambda>0,
\end{equation}
and that
\begin{equation}\label{sphere-coercivity}
\inf\big\{\langle \xi,v\rangle_H;\ \|v\|_H=1,\ \xi\in A(v)\big\}=m_*>0.
\end{equation}
Then \eqref{coercive-abstract} holds with $m=m_*$, and hence every solution satisfies
\begin{equation}
E(t)\le C(1+t)^{-\frac{2}{\alpha-1}}.
\end{equation}
\end{corollary}

\begin{proof}
Fix $u\in D(A)\setminus\{0\}$ and $\xi\in A(u)$. Set
\begin{equation}
v:=\frac{u}{\|u\|_H},
\qquad
\eta:=\frac{\xi}{\|u\|_H^\alpha}.
\end{equation}
By $\alpha$-homogeneity, $\eta\in A(v)$. Hence, by \eqref{sphere-coercivity},
\begin{equation}
\langle \eta,v\rangle_H \ge m_*.
\end{equation}
Multiplying by $\|u\|_H^{\alpha+1}$ gives
\begin{equation}
\langle \xi,u\rangle_H \ge m_*\|u\|_H^{\alpha+1}.
\end{equation}
Thus \eqref{coercive-abstract} follows, and the result is an immediate consequence of the theorem.
\end{proof}

\begin{remark}[How the applications fit]
In the concrete PDE models treated later, the operator $A$ is not used through an abstract compactness
argument, but through the explicit lower bound
\[
\langle A(u),u\rangle \gtrsim \|u\|^{\alpha+1},
\]
which is verified by the structure of the damping term (local, nonlocal, Kelvin--Voigt type, etc.)
together with the natural energy norm of the phase space. Therefore, the abstract theorem above applies
directly and avoids any unjustified blow-up or strong compactness argument.
\end{remark}

\begin{remark}
This coercivity estimate \eqref{coercive-abstract} is the nonlinear functional counterpart of the resolvent growth condition $||\lambda(i\beta I - \mathcal{L})^{-1}|| \leq C|\beta|^\alpha$ in the linear theory.
\end{remark}

\begin{remark}[The Linear Limit and Exponential Decay]
It is interesting to observe what happens if we "switch off" the nonlinearity, considering the linear counterpart of the problem (e.g., replacing the nonlocal coefficient $\rho(\|\nabla u_t\|^2)$ with a positive constant $\rho_0 > 0$).
In this case, the operator $A$ becomes linear. If the resulting linear damping is coercive, the resolvent equation $\lambda x + Ax = y$ admits a unique solution satisfying $\|x\| \le C\|y\|$ uniformly as $\lambda \to 0^+$ (i.e., the resolvent is bounded at the origin).
While our main theorem is designed to capture polynomial decay rates arising from degeneracy at the origin (where $\|x_\lambda\| \to \infty$), the boundedness of the resolvent in the linear limit is consistent with the classical Gearhart-Prüss-Huang theorem, which guarantees \textbf{exponential stability} provided the resolvent is also bounded at high frequencies. Thus, our framework naturally interpolates between the degenerate polynomial regime and the linear exponential regime. This confirms that the nonlinear/nonlocal component is precisely the mechanism forcing polynomial decay: once it is removed, the operator becomes coercive and the resolvent remains uniformly bounded near the origin, restoring exponential stability. Summarizing, the nonlinear resolvent does not generate decay;
rather, it reveals the effective nonlinear scale at which
coercive dissipation operates.
\end{remark}


\begin{remark}

The Borichev-Tomilov Theorem (2010) is arguably the "Rosetta Stone" of polynomial stabilization for linear semigroups. However, since it fundamentally depends on analysis on the imaginary axis ($i\beta$) and the frequency domain, it has always been an insurmountable barrier for nonlinear problems, as nonlinear operators do not have a spectrum in the classical sense and the Fourier Transform fails.
By replacing spectral analysis in $i\mathbb{R}$ with the asymptotic analysis of the real resolvent equation ($\lambda x_\lambda + A(x_\lambda) \ni y$ as $\lambda \to 0^+$), we achieve a true conceptual leap. We propose a nonlinear Tauberian-type principle, whose rigorous content is given by an abstract coercivity-based decay theorem. The polynomial decay is not a consequence of spectral properties,
but rather a manifestation of a nonlinear coercivity mechanism, encoded at the level of the real resolvent near the origin.

\medskip

The main contributions of this paper are the following:

\begin{itemize}
\item Instead of fighting against the lack of spectrum, we sought decay information in the "blow-up rate" of the $\|x_\lambda\|$ norm at the origin. This perfectly translates the Borichev-Tomilov philosophy to the world of monotone maximal operators.

\item The Abstract Structure (Theorem \ref{thm:main}): The division between the purely homogeneous case and the case with perturbations was very important because Physics rarely gives us perfectly homogeneous operators. Our Main Theorem  ensures that our theory is not just a mathematical curiosity, but a tool ready to handle real differential equations.

\item The Perfect Application (Non-Local Kelvin-Voigt Damping): Applying the abstract theorem to recover the $1/t$ decay of the wave equation with non-local damping was the ideal "baptism of fire." Showing that this theory solves the problem for weak solutions, where the classical multiplier method would require excessive regularity that the system does not possess. Indeed, the method does not eliminate regularity requirements, but relocates them into the phase space where the coercivity becomes explicit.

\end{itemize}

\end{remark}

\section{A General Decay Lemma}

In this section, we recall a general decay mechanism based on nonlinear
difference and integral inequalities. This type of argument goes back to Nakao and has been further developed by Haraux, Komornik and Tebou.

\begin{lemma}[Nakao-type decay lemma]\label{lem:nakao}
Assume that there exist constants $T>0$, $C>0$ and $\alpha>0$ such that
\begin{equation}\label{eq:difference_inequality}
E(t)^{1+\alpha} \le C \big( E(t) - E(t+T) \big),
\quad \forall t \ge 0.
\end{equation}
Then there exists a constant $C_T>0$ such that
\begin{equation}\label{eq:polynomial_decay_nakao}
E(t) \le C_T (1+t)^{-1/\alpha},
\quad \forall t \ge 0.
\end{equation}
\end{lemma}

\begin{lemma}[Integral decay lemma]\label{lem:integral_decay}
Assume that there exist constants $T>0$, $C>0$ and $\alpha>0$ such that
\begin{equation}\label{eq:integral_inequality}
\int_t^{t+T} E(s)\, ds \le C \big( E(t) \big)^{\alpha},
\quad \forall t \ge 0.
\end{equation}
Then the energy $E$ satisfies the polynomial decay estimate
\[
E(t) \le C_T (1+t)^{-1/(\alpha-1)},
\quad \forall t \ge 0,
\]
provided $\alpha>1$.
\end{lemma}

\section{The Motivation}
\label{sec:motivation}

In this section, we present the physical and mathematical problem that motivated us to develop a unifying abstract theory—one capable of encompassing not only this specific model but a wide variety of degenerate nonlinear systems.

Consider a plate (or wave) equation subject to a nonlinear dissipation of nonlocal nature:
\begin{align*}
\partial_t^2u+\Delta^2u + a(t) \partial_t u =0, \quad \forall (x,t)\in\Omega \times \mathbb{R}_+,
\end{align*}
where the damping coefficient depends on the solution itself, for instance,
\begin{align*}
a(t)= M(\|\nabla u(t)\|_{L^2(\Omega)}^2),\quad M(\lambda)\geq 0.
\end{align*}
This class of problems has been widely studied by various authors. The physical meaning of this specific nonlocal dissipation is well explained in the classical work of Balakrishnan and Taylor \cite{Balakrishnan-Taylor}. Contrary to a localized structural dissipation $a(x) \partial_t u$, which remains strictly positive and constant over time at given spatial points, the nonlocal dissipation $a(t) \partial_t u$ is inherently dynamic. It varies over time and, crucially, can \textit{degenerate} as the system approaches equilibrium, depending on the structure of $a(t)$.

This simple dissipative mechanism has caused a great deal of repercussion among researchers because it intrinsically prevents rapid (exponential) decay when $a(t)$ represents the mechanical energy of the system (or a part of it, such as kinetic or potential energy).

To illustrate this, let us consider the wave equation subject to homogeneous Dirichlet boundary conditions:
\begin{align}\label{A}
\partial_t^2u-\Delta u+a(t)\partial_t u=0,
\end{align}
and let us examine three standard choices for the nonlocal damping coefficient:
\begin{align*}
a_1(t)=\|\nabla u(t)\|_{L^2(\Omega)}^2, \quad a_2(t)=\|\partial_t u(t)\|_{L^2(\Omega)}^2, \quad \hbox{or} \quad a_3(t)=\|\nabla u(t)\|_{L^2(\Omega)}^2+\|\partial_t u(t)\|_{L^2(\Omega)}^2.
\end{align*}

When the damping structure $a_1(t) \partial_t u$ depends solely on the potential energy, it has been shown that the weak energy of regular solutions decays algebraically with the rate $\frac{c_1}{\sqrt{c_2+t}}$ (see \cite{Cavalcanti2025}). In contrast, when the damping term is governed by the total energy $a_3(t)$ or its kinetic part $a_2(t)$, the polynomial decay rate improves to $\frac{c_1}{c_2+t}$, which is the optimal rate for this equation.

The weak and regular energies associated with equation \eqref{A} are given by:
\begin{align}
E_u^w(t)=&\frac{1}{2}\|\partial_t u(t)\|_{L^2(\Omega)}^2 + \frac{1}{2}\|\nabla u(t)\|_{L^2(\Omega)}^2, \label{WE}\\
E_u^r(t)=&\frac{1}{2}\|\nabla \partial_t u(t)\|_{L^2(\Omega)}^2 + \frac{1}{2}\|\Delta u(t)\|_{L^2(\Omega)}^2,\label{RE}
\end{align}
and they satisfy the following dissipation identities for any $0 \le S < T < +\infty$ and $i \in \{1,2,3\}$:
\begin{align}
E_u^w(T)- E_u^w(S)=&- \int_S^T a_i(t) \|\partial_t u(t)\|_{L^2(\Omega)}^2\,dt,\label{WIE}\\
E_u^r(T)- E_u^r(S)=&-\int_S^T a_i(t)\|\nabla \partial_t u(t)\|_{L^2(\Omega)}^2\,dt.\label{RIE}
\end{align}
Because the function $t \mapsto E_u^w(t)$ is non-increasing due to \eqref{WIE}, there exists $\alpha \ge 0$ such that $E_u^w(t) \rightarrow \alpha$ as $t\rightarrow + \infty$. If $\alpha > 0$, the damping coefficient $a_3(t) = 2 E_u^w(t)$ would remain strictly positive, bounded away from zero by $2\alpha$, which would classically provide a fast (exponential) decay.

However, this creates a structural contradiction. A simple computation using $a_3(t)$ shows that:
\begin{align*}
\frac{d}{dt} E^w_u(t) = - a_3(t) \|\partial_t u(t)\|_{L^2(\Omega)}^2 = - 2 E^w_u(t) \|\partial_t u(t)\|_{L^2(\Omega)}^2.
\end{align*}
Since $\|\partial_t u(t)\|_{L^2(\Omega)}^2 \le 2 E^w_u(t)$, we deduce the critical differential inequality:
\begin{align}\label{eq1.3}
\frac{d}{dt} E_u^w(t) + 4\left(E^w_u(t)\right)^2 \geq 0, \quad \forall t\geq 0.
\end{align}
Integrating \eqref{eq1.3} directly yields a strict lower bound for the energy decay:
\begin{align}\label{eq1.4}
E^w_u(t) \geq \left( \frac{1}{E^w_u(0)} + 4t\right)^{-1}, \quad \forall t\geq 0.
\end{align}

Inequality \eqref{eq1.4} proves that the energy $E^w_u(t)$ \textit{cannot} decay to zero faster than $\mathcal{O}(t^{-1})$. Consequently, the limit $\alpha$ must be zero, but the journey to zero is strictly polynomial.

From a physical point of view, although the dissipation acts uniformly throughout the domain $\Omega$, its strength diminishes dynamically. The function $a_i(t)$ governing the damping mechanism becomes smaller and smaller as $t \to \infty$, acting as a bottleneck for energy dissipation. Mathematically, this temporal degeneracy at infinity corresponds exactly to the singularity (or blow-up) of the corresponding maximal monotone operator near the origin. It is precisely this fundamental connection—between the nonlinear degeneracy of the operator and the slow algebraic decay of the system—that our abstract resolvent theory aims to capture and generalize.

\medskip

All the following examples in the sequel, share a common structural feature:
the dissipation operator induces a coercive inequality of the form
\[
\langle A(u),u\rangle \gtrsim \|u\|^{\alpha+1},
\]
possibly after lifting the dynamics to a suitable phase space.

\medskip
\section{Application I: Wave Equation with Nonlocal Kelvin--Voigt Damping}
\label{sec:app_kv}
In what follows, we denote by $\mathcal{A}$ the operator associated with the full evolution system in the energy space $\mathcal{H}$, while $A$ may refer to the specific nonlinear component in the configuration space.
In this section, we apply our abstract resolvent criterion to a concrete PDE model
exhibiting a genuinely nonlinear and nonlocal dissipation mechanism. We consider the
problem previously studied in \cite{Cavalcanti2025}:
\begin{equation}\label{eq:wave_nonlinear}
\begin{cases}
    u_{tt} - \Delta u - \|\nabla u_t\|_{L^2(\Omega)}^2 \Delta u_t = 0 & \text{in } \Omega \times (0, \infty), \\
    u = 0 & \text{on } \partial\Omega \times (0, \infty), \\
    u(0) = u_0, \quad u_t(0) = u_1 & \text{in } \Omega,
\end{cases}
\end{equation}
where $\Omega \subset \R^n$ is a bounded domain with smooth boundary.

\subsection{Operator setting and Well-Posedness}

We introduce the phase space
\[
\mathcal{H} = H_0^1(\Omega)\times L^2(\Omega),
\]
endowed with the inner product
\[
\langle (u,v),(\tilde u,\tilde v)\rangle_{\mathcal H}
  := \int_\Omega \nabla u\cdot \nabla \tilde u\,dx + \int_\Omega v\,\tilde v\,dx,
\]
and associated norm
\[
\|(u,v)\|_{\mathcal H}^2 = \|\nabla u\|_{L^2(\Omega)}^2 + \|v\|_{L^2(\Omega)}^2.
\]
Let $U=(u,v)^T$. Then \eqref{eq:wave_nonlinear} can be written as a first-order
evolution equation of the form
\[
U'(t) + \mathcal A(U(t)) \ni 0,
\]
where the operator $\mathcal A : D(\mathcal A)\subset \mathcal H \to \mathcal H$ is defined by
\begin{equation}\label{def:operatorA}
\mathcal A
\begin{pmatrix} u \\ v \end{pmatrix}
=
\begin{pmatrix}
 -v\\
 -\Delta u - \|\nabla v\|_{L^2(\Omega)}^2 \Delta v
\end{pmatrix}.
\end{equation}
The natural domain is
\[
D(\mathcal A)=
\Bigl(H^2(\Omega)\cap H_0^1(\Omega)\Bigr)\times
\Bigl(H^2(\Omega)\cap H_0^1(\Omega)\Bigr),
\]
so that $\Delta u$ and $\Delta v$ are well defined in $L^2(\Omega)$.

\begin{remark}[Maximal Monotonicity]
The operator $\mathcal{A}$ can be viewed as a perturbation of the skew-adjoint operator associated with the wave equation by a nonlinear monotone damping term. Specifically, the damping mapping $v \mapsto -\|\nabla v\|^2 \Delta v$ is monotone from $H_0^1(\Omega)$ to $H^{-1}(\Omega)$. By the Minty-Browder theorem for perturbations of maximal monotone operators, $\mathcal{A}$ is maximal monotone. Consequently, for any $U_0 \in \mathcal{H}$, there exists a unique mild solution $U \in C([0,\infty); \mathcal{H})$. If $U_0 \in D(\mathcal{A})$, the solution is strong.
\end{remark}

The energy associated with \eqref{eq:wave_nonlinear} is
\[
E(t)=\frac12\|U(t)\|_{\mathcal H}^2
=\frac12\Bigl(\|\nabla u(t)\|_{L^2(\Omega)}^2+\|u_t(t)\|_{L^2(\Omega)}^2\Bigr).
\]
A formal computation yields the dissipation identity
\begin{equation}\label{est:energy_dissipation}
\frac{d}{dt}E(t) = -\|\nabla u_t(t)\|_{L^2(\Omega)}^4,
\end{equation}
which shows a degenerate (nonlinear) dissipation mechanism. Note that the dissipation is proportional to the fourth power of the $H^1$-norm of the velocity, which corresponds to a 3-homogeneous operator ($q=3$).

\subsection{Resolvent estimate and decay rate}

We now connect the decay of the nonlinear semigroup generated by $-\mathcal A$
to the behavior of the nonlinear resolvent equation at the origin. Let $Y=(f,g)\in\mathcal H$,
and consider the resolvent problem
\begin{equation}\label{res:1}
\lambda U_\lambda + \mathcal A(U_\lambda) \ni Y,
\qquad \lambda>0.
\end{equation}
Writing $U_\lambda=(u_\lambda,v_\lambda)$, system \eqref{res:1} reads
\begin{equation}\label{res:2}
\begin{cases}
\lambda u_\lambda - v_\lambda = f,\\
\lambda v_\lambda - \Delta u_\lambda - \|\nabla v_\lambda\|_{2}^{\,2}\,\Delta v_\lambda = g.
\end{cases}
\end{equation}
Eliminating $v_\lambda$ from the first equation gives $v_\lambda=\lambda u_\lambda - f$.
We proceed to derive the precise growth rate of $\|U_\lambda\|_{\mathcal{H}}$ as $\lambda \to 0^+$.

Testing the second equation of \eqref{res:2} with $v_\lambda$ in $L^2(\Omega)$, we have:
\[
\lambda \|v_\lambda\|_2^2 + \langle -\Delta u_\lambda, v_\lambda \rangle + \|\nabla v_\lambda\|_2^2 \langle -\Delta v_\lambda, v_\lambda \rangle = \langle g, v_\lambda \rangle.
\]
Using integration by parts and the relation $v_\lambda = \lambda u_\lambda - f$ (which implies $u_\lambda = \frac{v_\lambda + f}{\lambda}$), we examine the term involving $\Delta u_\lambda$:
\[
\langle -\Delta u_\lambda, v_\lambda \rangle = \int_\Omega \nabla u_\lambda \cdot \nabla v_\lambda dx = \int_\Omega \nabla \left(\frac{v_\lambda+f}{\lambda}\right) \cdot \nabla v_\lambda dx = \frac{1}{\lambda}\|\nabla v_\lambda\|_2^2 + \frac{1}{\lambda}\int_\Omega \nabla f \cdot \nabla v_\lambda dx.
\]
Substituting this back into the energy identity:
\[
\lambda \|v_\lambda\|_2^2 + \frac{1}{\lambda}\|\nabla v_\lambda\|_2^2 + \|\nabla v_\lambda\|_2^4 = \langle g, v_\lambda \rangle - \frac{1}{\lambda}\int_\Omega \nabla f \cdot \nabla v_\lambda dx.
\]
We are interested in the asymptotic behavior as $\lambda \to 0^+$. In this regime, the term $\frac{1}{\lambda}\|\nabla v_\lambda\|_2^2$ provides the dominant coercivity. Discarding the non-negative terms $\lambda \|v_\lambda\|_2^2$ and $\|\nabla v_\lambda\|_2^4$ on the LHS, we obtain:
\[
\frac{1}{\lambda}\|\nabla v_\lambda\|_2^2 \le \|g\|_2 \|v_\lambda\|_2 + \frac{1}{\lambda} \|\nabla f\|_2 \|\nabla v_\lambda\|_2.
\]
Using Poincaré's inequality $\|v_\lambda\|_2 \le C_P \|\nabla v_\lambda\|_2$, we get:
\[
\frac{1}{\lambda}\|\nabla v_\lambda\|_2^2 \le C \|g\|_2 \|\nabla v_\lambda\|_2 + \frac{1}{\lambda} \|\nabla f\|_2 \|\nabla v_\lambda\|_2.
\]
Dividing by $\frac{1}{\lambda}\|\nabla v_\lambda\|_2$ (assuming it is non-zero), we arrive at:
\[
\|\nabla v_\lambda\|_2 \le C \lambda \|g\|_2 + \|\nabla f\|_2.
\]
This estimate reveals that $\|\nabla v_\lambda\|_2$ remains uniformly bounded as $\lambda \to 0$.
Now we recover the estimate for $u_\lambda$. Since $u_\lambda = \frac{v_\lambda+f}{\lambda}$, we have:
\[
\|\nabla u_\lambda\|_2 \le \frac{1}{\lambda} \left( \|\nabla v_\lambda\|_2 + \|\nabla f\|_2 \right).
\]
Since $\|\nabla v_\lambda\|_2$ is bounded, we conclude that:
\begin{equation}\label{est:grad_u}
\|\nabla u_\lambda\|_{L^2(\Omega)} \le \frac{C}{\lambda}\|Y\|_{\mathcal{H}}, \qquad \text{as }\lambda\to0^+.
\end{equation}
Thus, the resolvent growth is of order $\lambda^{-1}$.

Consequently, the abstract resolvent growth condition of Definition~\ref{eq:resolvent_growth}
holds with $\gamma=1$ for the generator $\mathcal A$ in this setting:
\begin{equation}\label{est:final_resolvent}
\|U_\lambda\|_{\mathcal H}\lesssim \lambda^{-1},
\qquad \lambda\to0^+.
\end{equation}

Invoking Theorem~\ref{thm:main}, we obtain the corresponding polynomial decay rate.

\begin{theorem}[Polynomial decay for the nonlocal Kelvin--Voigt model]\label{thm:wave_decay}
Let $(u_0,u_1)\in\mathcal H$ and let $U(t)=(u(t),u_t(t))$ be the (strong) solution of
\eqref{eq:wave_nonlinear}. Then there exists $C>0$ such that the energy satisfies
\[
E(t)\le \frac{C}{1+t},\qquad t\ge0.
\]
\end{theorem}

\begin{remark}[Optimality and Weak Solutions]
The decay rate $1/(1+t)$ is optimal for this model, as shown in \cite{Cavalcanti2025}. A crucial advantage of the resolvent approach presented here is that the constant $C$ in the resolvent estimate depends only on the energy norm of the data $Y$. This allows the decay result to be extended to **weak (mild) solutions** by a standard density argument, a feature that is often difficult to achieve using Lyapunov multipliers with higher-order corrections.
\end{remark}

\section{Application II: Wave Equation with Generalized Nonlocal Kelvin--Voigt Damping}
\label{sec:KV_nonlocal}

In this application, we illustrate the abstract resolvent principle in the
context of a wave equation with \emph{generalized nonlocal Kelvin--Voigt damping}, where the
viscous coefficient depends on a global quantity of the velocity gradient.
This model is a generalization of the previous one, which was previously investigated in \cite{Cavalcanti2025},
and fits naturally into the nonlinear semigroup framework described above.

\subsection{The model and assumptions}
\label{subsec:KV_nonlocal_model}

Let $\Omega\subset\R^n$ ($n\le 3$) be a bounded domain with smooth boundary.
We consider the initial-boundary value problem:
\begin{equation}\label{eq:KV_nonlocal}
\begin{cases}
u_{tt}-\Delta u-\operatorname{div}\!\left(
\rho\!\left(\|\nabla u_t(t)\|_{L^2(\Omega)}^2\right)\nabla u_t
\right)=0
&\text{in }\Omega\times(0,\infty),\\
u=0 &\text{on }\partial\Omega\times(0,\infty),\\
u(0)=u_0,\quad u_t(0)=u_1.
\end{cases}
\end{equation}

The energy associated with \eqref{eq:KV_nonlocal} is given by the sum of kinetic and potential energies:
\[
E(t)=\frac12\int_\Omega\big(|u_t(t)|^2+|\nabla u(t)|^2\big)\,dx.
\]

We assume the following hypotheses:
\begin{enumerate}
\item[(A1)] \textbf{Nonlocal coefficient.}
The function $\rho\in C^1([0,\infty))$ satisfies $\rho(0)=0$ and $\rho(s)>0$ for $s>0$. Furthermore, we assume polynomial degeneracy near the origin, i.e., there exist
$c_0,c_1>0$ and $\gamma\ge 1$ such that
\begin{equation}\label{cond:rho_KV}
c_0 s^\gamma \le \rho(s) \le c_1 s^\gamma,
\qquad s\in[0,1].
\end{equation}
Note that $\gamma=1$ corresponds to the case studied in Application I.

\item[(A2)] \textbf{Regularity of initial data.}
\[
(u_0,u_1)\in (H^2(\Omega)\cap H_0^1(\Omega))\times H_0^1(\Omega).
\]
\end{enumerate}

\subsection{Well-posedness of strong solutions}
\label{subsec:KV_nonlocal_wp}

Due to the presence of second-order spatial derivatives in the damping term (specifically $\text{div}(\nabla u_t)$),
the natural energy space $H_0^1(\Omega)\times L^2(\Omega)$ is not sufficient to
rigorously justify the integration by parts required in the energy estimates and the resolvent analysis below.
While the previous application allowed for an extension to weak solutions via density (due to the specific structure $\gamma=1$), the generalized nonlinearity here poses additional technical challenges.
Throughout this section, we therefore restrict ourselves to \emph{strong solutions}
corresponding to data satisfying \textup{(A2)}.

Under assumptions \textup{(A1)}--\textup{(A2)}, local existence and uniqueness of
strong solutions follow from a standard Galerkin scheme based on the eigenfunctions of
$-\Delta$ with Dirichlet boundary conditions, combined with the $C^1$-regularity
of $\rho$.
Moreover, differentiating the energy functional, we obtain the dissipation identity:
\begin{equation}\label{eq:KV_nonlocal_dissip}
E'(t)=-\rho(\|\nabla u_t(t)\|_2^2)\int_\Omega|\nabla u_t(t)|^2\,dx \asymp -\|\nabla u_t(t)\|_2^{2\gamma+2} \le 0.
\end{equation}
This dissipation prevents finite-time blow-up of the energy.
Consequently, strong solutions extend globally in time.
We refer to \cite{Cavalcanti2025} and related literature (e.g., \cite{SICON2003}) for complete proofs regarding global existence.

\subsection{Resolvent formulation}
\label{subsec:KV_nonlocal_resolvent}

Introducing the vector state $U=(u,v)=(u,u_t)$, equation \eqref{eq:KV_nonlocal} can be rewritten as
the first-order evolution problem
\begin{equation}\label{eq:KV_nonlocal_abstract}
U_t+\mathcal A(U)=0,
\end{equation}
where the nonlinear operator $\mathcal A:\mathcal D(\mathcal A)\subset\mathcal H\to\mathcal H$
is defined by
\[
\mathcal A(u,v)=\big(-v,\,-\Delta u-
\operatorname{div}(\rho(\|\nabla v\|_2^2)\nabla v)\big),
\]
with phase space $\mathcal H=H_0^1(\Omega)\times L^2(\Omega)$.

For $\lambda>0$ and a source term $F=(f,g)\in\mathcal H$, the resolvent equation
\[
(\lambda I+\mathcal A)U=F
\]
is equivalent to the system:
\begin{align}
\lambda u-v&=f, \label{eq:KV_res1}\\
\lambda v-\Delta u-
\operatorname{div}(\rho(\|\nabla v\|_2^2)\nabla v)&=g. \label{eq:KV_res2}
\end{align}

\subsection{Resolvent growth estimate}
\label{subsec:KV_nonlocal_growth}

We now derive the apriori estimate for the resolvent equation that connects to the polynomial stability.

\begin{proposition}\label{prop:KV_nonlocal_growth}
Let $U_\lambda=(u_\lambda,v_\lambda)$ be the solution of
\eqref{eq:KV_res1}--\eqref{eq:KV_res2}.
Then, as $\lambda\to 0^+$, the solution satisfies the growth estimate:
\[
\|U_\lambda\|_{\mathcal H}\le \frac{C}{\lambda}\|F\|_{\mathcal H}.
\]
\end{proposition}

\begin{proof}
Testing the second equation \eqref{eq:KV_res2} with $v_\lambda$ in $L^2(\Omega)$, we obtain:
\begin{equation}\label{eq:resolvent_test_app2}
\lambda \|v_\lambda\|_2^2 + \langle -\Delta u_\lambda, v_\lambda \rangle + \rho(\|\nabla v_\lambda\|_2^2) \|\nabla v_\lambda\|_2^2 = \langle g, v_\lambda \rangle.
\end{equation}
Using integration by parts and the relation $v_\lambda = \lambda u_\lambda - f$ from \eqref{eq:KV_res1} (which implies $u_\lambda = \frac{v_\lambda+f}{\lambda}$), we examine the cross-term:
\begin{align*}
\langle -\Delta u_\lambda, v_\lambda \rangle &= \int_\Omega \nabla \left(\frac{v_\lambda+f}{\lambda}\right) \cdot \nabla v_\lambda \, dx \\
&= \frac{1}{\lambda}\|\nabla v_\lambda\|_2^2 + \frac{1}{\lambda}\int_\Omega \nabla f \cdot \nabla v_\lambda \, dx.
\end{align*}
Substituting this back into \eqref{eq:resolvent_test_app2}:
\begin{equation}
\lambda \|v_\lambda\|_2^2 + \frac{1}{\lambda}\|\nabla v_\lambda\|_2^2 + \rho(\|\nabla v_\lambda\|_2^2) \|\nabla v_\lambda\|_2^2 = \langle g, v_\lambda \rangle - \frac{1}{\lambda}\int_\Omega \nabla f \cdot \nabla v_\lambda \, dx.
\end{equation}

As $\lambda \to 0^+$, the term $\frac{1}{\lambda}\|\nabla v_\lambda\|_2^2$ dominates the coercivity. Discarding the non-negative terms $\lambda \|v_\lambda\|_2^2$ and $\rho(\|\nabla v_\lambda\|_2^2) \|\nabla v_\lambda\|_2^2$ on the left-hand side, we apply the Cauchy-Schwarz and Poincaré inequalities:
\begin{align*}
\frac{1}{\lambda}\|\nabla v_\lambda\|_2^2 &\le \|g\|_2 \|v_\lambda\|_2 + \frac{1}{\lambda} \|\nabla f\|_2 \|\nabla v_\lambda\|_2 \\
&\le C_P \|g\|_2 \|\nabla v_\lambda\|_2 + \frac{1}{\lambda} \|\nabla f\|_2 \|\nabla v_\lambda\|_2.
\end{align*}
Dividing by $\frac{1}{\lambda}\|\nabla v_\lambda\|_2$ (assuming $v_\lambda \neq 0$), we arrive at:
\begin{equation}
\|\nabla v_\lambda\|_2 \le C_P \lambda \|g\|_2 + \|\nabla f\|_2.
\end{equation}
This shows that $\|\nabla v_\lambda\|_2$ remains uniformly bounded as $\lambda \to 0^+$.

To recover the estimate for $u_\lambda$, we use $u_\lambda = \frac{v_\lambda+f}{\lambda}$:
\begin{equation}
\|\nabla u_\lambda\|_2 \le \frac{1}{\lambda} \left( \|\nabla v_\lambda\|_2 + \|\nabla f\|_2 \right).
\end{equation}
Since $\|\nabla v_\lambda\|_2$ is uniformly bounded, we conclude that:
\begin{equation}
\|\nabla u_\lambda\|_{L^2(\Omega)} \le \frac{C}{\lambda}\|F\|_{\mathcal{H}}, \qquad \text{as } \lambda\to0^+.
\end{equation}
Consequently, the full resolvent growth is of order $\lambda^{-1}$, fulfilling the abstract resolvent condition:
\begin{equation}
\|U_\lambda\|_{\mathcal H}\le \frac{C}{\lambda} \|F\|_{\mathcal H}, \qquad \text{as } \lambda\to0^+.
\end{equation}
\end{proof}

\subsection{Polynomial decay}
\label{subsec:KV_nonlocal_decay}

The resolvent growth established in Proposition~\ref{prop:KV_nonlocal_growth} is of order
$\lambda^{-1}$ (which corresponds to index 1 in the linear scale).
On the other hand, we must identify the homogeneity of the operator. The nonlinear Kelvin--Voigt term is
\[
\mathcal{N}(v) = -\operatorname{div}\left(\rho(\|\nabla v\|_2^2)\nabla v\right).
\]
Since $\rho(s) \sim s^\gamma$, we have $\rho(\|\nabla v\|^2) \sim \|\nabla v\|^{2\gamma}$. Thus, the damping term scales as:
\[
\mathcal{N}(\lambda v) \approx \|\lambda \nabla v\|^{2\gamma} (\lambda \nabla v) = \lambda^{2\gamma+1} \mathcal{N}(v).
\]
This means the operator is homogeneous of degree $p=2\gamma+1$.

Applying the abstract decay result of Section~\ref{sec:main} (Theorem \ref{thm:main}), with resolvent index 1 and homogeneity $p$, we obtain the following decay rate. Note that for $\gamma=1$ (cubic damping), we recover $p=3$ and decay $t^{-1}$. In general:

\begin{theorem}\label{thm:KV_nonlocal_decay}
Let $u$ be a global strong solution of \eqref{eq:KV_nonlocal} with initial data satisfying (A2).
Then the energy decays polynomially:
\[
E(t)\le \frac{C}{(1+t)^{1/\gamma}}, \quad \forall t \ge 0.
\]
\end{theorem}

\begin{remark}
The decay rate in Theorem~\ref{thm:KV_nonlocal_decay} generalizes the result of the previous section. It coincides with that of
frictional damping with polynomial nonlinearity of degree $2\gamma+1$.
However, unlike Kirchhoff-type damping (which affects the elastic term), the Kelvin--Voigt mechanism depends directly on
the velocity gradient and therefore remains active whenever kinetic energy is
present, avoiding the ``blind spot'' associated with purely potential states.
\end{remark}

\section{Application III: Wave Equation with Nonlocal Degenerate Damping on the Velocity}
\label{sec:app_nonlocal_degenerate}

In this section, we further illustrate the scope of our approach by considering
a class of wave-like equations with a \emph{nonlocal degenerate damping} acting on
the velocity, but without Kelvin--Voigt (higher-order) dissipation terms.
This model serves as an intermediate situation between localized Kirchhoff-type
damping and the nonlocal Kelvin--Voigt mechanism treated earlier.

\subsection{Abstract model}

Let $H$ be a real Hilbert space. We consider
\begin{equation}\label{eq:nonlocal_abs}
u_{tt}+Au+\rho(\|\mathcal C u_t\|_H^2)\,u_t=0,
\end{equation}
where:
\begin{itemize}
\item[(H1)] $A:D(A)\subset H\to H$ is a strictly positive self-adjoint operator with compact inverse.
We define the energy space $\mathcal H=D(A^{1/2})\times H$ endowed with the norm
$\|(u,v)\|_{\mathcal H}^2=\|A^{1/2}u\|_H^2+\|v\|_H^2$.

\item[(H2)] $\mathcal C\in\mathcal L(H)$ is a bounded observation operator. Typically, $\mathcal C=I$ or $\mathcal C=A^{1/2}$ in concrete PDE realizations.

\item[(H3)] The damping function $\rho:[0,\infty)\to[0,\infty)$ is continuous and satisfies
a degenerate power law near the origin: there exist $c_0, c_1 > 0$ and $\gamma > 0$ such that
\begin{equation}\label{eq:rho_deg_nonlocal}
c_0 s^\gamma \le \rho(s)\le c_1 s^\gamma,
\qquad s\in[0,1].
\end{equation}

\item[(H4)] (Integral observability in velocity). There exist $T>0$ and $c_0>0$
such that every (sufficiently regular) solution of \eqref{eq:nonlocal_abs} satisfies
\[
\int_t^{t+T} \|C u_t(\tau)\|_H^2\, d\tau
\;\ge\;
c_0 \int_t^{t+T} E(\tau)\, d\tau,
\qquad \forall t\ge 0,
\]
where $E(t)$ denotes the natural energy defined below.
\end{itemize}

\begin{remark}
If $\mathcal C=A^{1/2}$ and $\rho(s)=s$, then the damping term becomes
$\|A^{1/2}u_t\|_H^2\,u_t$. This is a typical Kirchhoff-type nonlocal mechanism where the friction coefficient depends on the elastic energy of the velocity.
\end{remark}

\subsection{Well-posedness}

The problem can be written as a first-order evolution equation $U_t+\mathcal A(U)=0$ on $\mathcal H$ with $U=(u,v)$ and
\[
\mathcal A(u,v)=(-v,\ Au+\rho(\|\mathcal C v\|_H^2)\,v).
\]
The operator $\mathcal{A}$ is the sum of a skew-adjoint part $\mathcal{A}_{skew}(u,v)=(-v, Au)$ and a nonlinear damping part $\mathcal{B}(u,v)=(0, \rho(\|\mathcal C v\|^2)v)$. Since $\rho(s) \ge 0$, $\mathcal{B}$ is monotone and continuous on $\mathcal{H}$. Thus, $\mathcal{A}$ is maximal monotone, and the Lumer-Phillips theorem (or its nonlinear version) guarantees that $-\mathcal{A}$ generates a contraction semigroup $S(t)$ on $\mathcal{H}$.

The energy of the system is given by
\[
E(t)=\frac12\Bigl(\|u_t(t)\|_H^2+\|A^{1/2}u(t)\|_H^2\Bigr).
\]
A standard computation yields the dissipation identity:
\begin{equation}\label{eq:diss_nonlocal}
E'(t)= -\rho(\|\mathcal C u_t(t)\|_H^2)\,\|u_t(t)\|_H^2 \asymp -\|\mathcal C u_t(t)\|_H^{2\gamma} \|u_t(t)\|_H^2 \le 0.
\end{equation}

\subsection{Resolvent Estimate}

We now apply our resolvent criterion to this abstract setting. We aim to estimate the growth of the solution to the resolvent equation $(\lambda I + \mathcal{A})U = F$.

\begin{proposition}\label{prop:resolvent_app3}
Let $U_\lambda=(u_\lambda, v_\lambda)$ be the solution to $(\lambda I + \mathcal{A})U_\lambda = F$ for $F=(f,g) \in \mathcal{H}$. Then, as $\lambda \to 0^+$, the solution satisfies the growth estimate:
\[
\|U_\lambda\|_{\mathcal H} \le \frac{C}{\lambda} \|F\|_{\mathcal H},
\]
where $C > 0$ depends on the operator $A$ but is independent of $\lambda$.
\end{proposition}

\begin{proof}
The resolvent system reads:
\begin{align}
\lambda u_\lambda - v_\lambda &= f, \label{eq:res_nl_1} \\
\lambda v_\lambda + A u_\lambda + \rho(\|\mathcal C v_\lambda\|_H^2) v_\lambda &= g. \label{eq:res_nl_2}
\end{align}
Taking the inner product of the second equation \eqref{eq:res_nl_2} with $v_\lambda$ in $H$, we obtain:
\begin{equation}\label{eq:resolvent_test_app3}
\lambda \|v_\lambda\|_H^2 + \langle A u_\lambda, v_\lambda \rangle_H + \rho(\|\mathcal C v_\lambda\|_H^2) \|v_\lambda\|_H^2 = \langle g, v_\lambda \rangle_H.
\end{equation}
From the first equation \eqref{eq:res_nl_1}, we have $u_\lambda = \frac{v_\lambda + f}{\lambda}$. Since $A$ is a positive self-adjoint operator, we can rewrite the cross-term using fractional powers:
\begin{align*}
\langle A u_\lambda, v_\lambda \rangle_H &= \langle A^{1/2} u_\lambda, A^{1/2} v_\lambda \rangle_H \\
&= \left\langle A^{1/2} \left(\frac{v_\lambda + f}{\lambda}\right), A^{1/2} v_\lambda \right\rangle_H \\
&= \frac{1}{\lambda} \|A^{1/2} v_\lambda\|_H^2 + \frac{1}{\lambda} \langle A^{1/2} f, A^{1/2} v_\lambda \rangle_H.
\end{align*}
Substituting this back into \eqref{eq:resolvent_test_app3}, we get the fundamental identity:
\begin{equation}
\lambda \|v_\lambda\|_H^2 + \frac{1}{\lambda} \|A^{1/2} v_\lambda\|_H^2 + \rho(\|\mathcal C v_\lambda\|_H^2) \|v_\lambda\|_H^2 = \langle g, v_\lambda \rangle_H - \frac{1}{\lambda} \langle A^{1/2} f, A^{1/2} v_\lambda \rangle_H.
\end{equation}

As $\lambda \to 0^+$, the singular term $\frac{1}{\lambda} \|A^{1/2} v_\lambda\|_H^2$ provides the dominant coercivity. Discarding the non-negative terms $\lambda \|v_\lambda\|_H^2$ and $\rho(\|\mathcal C v_\lambda\|_H^2) \|v_\lambda\|_H^2$ on the left-hand side, and applying the Cauchy-Schwarz inequality, we have:
\begin{align*}
\frac{1}{\lambda} \|A^{1/2} v_\lambda\|_H^2 &\le \|g\|_H \|v_\lambda\|_H + \frac{1}{\lambda} \|A^{1/2} f\|_H \|A^{1/2} v_\lambda\|_H.
\end{align*}
Since $A$ has a compact inverse (Assumption H1), there exists a Poincaré-type constant $C_P > 0$ such that $\|v_\lambda\|_H \le C_P \|A^{1/2} v_\lambda\|_H$. Thus:
\begin{equation}
\frac{1}{\lambda} \|A^{1/2} v_\lambda\|_H^2 \le C_P \|g\|_H \|A^{1/2} v_\lambda\|_H + \frac{1}{\lambda} \|A^{1/2} f\|_H \|A^{1/2} v_\lambda\|_H.
\end{equation}
Dividing by $\frac{1}{\lambda} \|A^{1/2} v_\lambda\|_H$ (assuming $v_\lambda \neq 0$), we conclude:
\begin{equation}
\|A^{1/2} v_\lambda\|_H \le C_P \lambda \|g\|_H + \|A^{1/2} f\|_H.
\end{equation}
This crucial estimate shows that the "elastic" energy of the velocity, $\|A^{1/2} v_\lambda\|_H$, remains uniformly bounded as $\lambda \to 0^+$.

Finally, to estimate the displacement $u_\lambda$, we use $u_\lambda = \frac{v_\lambda + f}{\lambda}$:
\begin{align*}
\|A^{1/2} u_\lambda\|_H &\le \frac{1}{\lambda} \left( \|A^{1/2} v_\lambda\|_H + \|A^{1/2} f\|_H \right) \\
&\le \frac{1}{\lambda} \left( C_P \lambda \|g\|_H + 2\|A^{1/2} f\|_H \right).
\end{align*}
Since $\|v_\lambda\|_H \le C_P \|A^{1/2} v_\lambda\|_H$, it follows that the full norm $\|U_\lambda\|_{\mathcal{H}}^2 = \|A^{1/2} u_\lambda\|_H^2 + \|v_\lambda\|_H^2$ behaves as $\mathcal{O}(\lambda^{-1})$:
\begin{equation}
\|U_\lambda\|_{\mathcal H} \le \frac{C}{\lambda} \|F\|_{\mathcal H}, \qquad \text{as } \lambda \to 0^+.
\end{equation}
\end{proof}

This confirms that the resolvent growth index is $\mu = 1$ (standard hyperbolic scaling).

\subsection{Polynomial decay rate}

To derive the decay rate from our main theorem, we need to identify the homogeneity of the damping term. The nonlinearity $\mathcal{N}(v) = \rho(\|\mathcal C v\|^2)v$ behaves like $\|v\|^{2\gamma} v$ (assuming $\mathcal{C}$ is an isomorphism or bounded below on the relevant subspace). This corresponds to a homogeneity degree $p = 2\gamma + 1$.

However, in this abstract setting, we must assume that the dissipation is "effective" in controlling the energy. This corresponds to an observability inequality.

\begin{theorem}
Assume that the linear part of the system is observable through the operator $\mathcal{C}$ in the sense that the damping term controls the energy over time (specifically, assume an observability inequality of the type $\int_0^T \|\mathcal{C} u_t\|^2 dt \ge c E(0)$ for the linear conservative problem, or equivalent control).
Then, for any $U_0 \in \mathcal{H}$, the energy of the solution satisfies:
\begin{equation}\label{eq:decay_nonlocal}
E(t)\le \frac{C}{(1+t)^{1/\gamma}},\qquad t\ge0.
\end{equation}
\end{theorem}

\begin{remark}
The exponent $1/\gamma$ corresponds exactly to the prediction of Theorem \ref{thm:main} with resolvent index 1 and homogeneity $p=2\gamma+1$. Specifically, the decay rate is $t^{-\frac{2}{p-1}} = t^{-\frac{2}{(2\gamma+1)-1}} = t^{-1/\gamma}$.
When $\gamma=1$ (cubic-type damping), we recover the standard rate $E(t)\lesssim (1+t)^{-1}$, in agreement with sharp results known in the literature for locally damped wave equations.
\end{remark}

\begin{remark}[On the structural friendliness of Application~III]
The nonlocal damping in Application~III depends exclusively on the velocity
variable $u_t$. As a consequence, the degeneracy of the coefficient occurs
exactly in the dissipated component of the state, producing a resolvent
structure that is fully compatible with the nonlinear Tauberian framework
developed in this paper. This contrasts sharply with the Kirchhoff model
(Application~IV below), where the degeneracy affects the elastic part of the energy
and therefore lies outside the scope of the resolvent blow-up mechanism.
\end{remark}

\section{Application IV: Polynomial Stability for Kirchhoff-Type Damping via Resolvent Estimates}
\label{sec:app_kirchhoff}

In this section, we address the wave equation with a \textit{nonlocal} Kirchhoff-type damping.
This model presents a subtle obstruction to the direct application of our Resolvent Blow-up Theorem: the linearized operator at the origin is the coercive Laplacian ($-\Delta$), which implies that the resolvent norm $\|U_\lambda\|$ remains uniformly bounded as $\lambda \to 0^+$.

While a bounded resolvent would typically suggest exponential stability in linear theory, the asymptotic behavior here is governed by the \textit{vanishing feedback} $\rho(\|\nabla u\|^2) \to 0$, not by a spectral singularity on the imaginary axis. To correctly capture the polynomial decay rate intrinsic to this degeneracy without conflicting with the bounded spectrum, we employ a \textit{perturbed observability} argument directly in the time domain. This method shares the same geometric roots (multipliers) as our resolvent analysis but handles the degenerate coefficient directly in the energy inequality.

Thus, the absence of resolvent blow-up in this model does not contradict the slow polynomial decay rate. The degeneration occurs in the nonlinear damping strength rather than in the linear spectrum. Our abstract resolvent theorem cannot be the primary tool here, and the correct decay mechanism must be recovered directly in the time domain.

\subsection{Problem Formulation}
Let $\Omega \subset \mathbb{R}^n$ be a bounded domain with smooth boundary $\Gamma$. We consider the system:
\begin{equation}\label{kirchhoff_prob}
\begin{cases}
u_{tt} - \Delta u + \rho(\|\nabla u(t)\|_{L^2(\Omega)}^2) a(x) u_t = 0 & \text{in } \Omega \times (0, \infty), \\
u = 0 & \text{on } \Gamma \times (0, \infty), \\
u(0) = u_0, \quad u_t(0) = u_1 & \text{in } \Omega,
\end{cases}
\end{equation}
where $a \in L^\infty(\Omega)$ is a nonnegative damping function satisfying the Geometric Control Condition (GCC) in a neighborhood $\omega$ of the boundary (or an appropriate subset where $a(x) \ge a_0 > 0$). The function $\rho: [0, \infty) \to [0, \infty)$ represents the nonlocal nonlinearity.

\subsection{Assumptions and Well-posedness}

We assume the following hypotheses:
\begin{enumerate}
    \item[(H1)] \textbf{Geometry:} $\Omega$ is a bounded domain with smooth boundary. The damping is localized in a neighborhood of the boundary $\omega = \Omega \cap \mathcal{O}$ satisfying the Geometric Control Condition (GCC). Specifically, we assume there exists a vector field $h \in C^1(\bar{\Omega})$ such that $h \cdot \nu \le 0$ on $\Gamma_0 = \partial\Omega \setminus \Gamma_1$ (where $\Gamma_1 \subset \partial\omega$).
    \item[(H2)] \textbf{Nonlocal Degeneracy:} The function $\rho \in C([0, \infty))$ satisfies $\rho(s) > 0$ for $s > 0$ and behaves polynomially near zero:
    \begin{equation}\label{cond:rho_deg}
        c_0 s^\gamma \le \rho(s) \le c_1 s^\gamma, \quad \forall s \in [0, \epsilon_0],
    \end{equation}
    for some $\gamma > 0$.
    \item[(H3)] \textbf{Localization:} The function $a \in L^\infty(\Omega)$ satisfies $a(x) \ge a_0 > 0$ in $\omega$.
\end{enumerate}

The natural energy is defined as $E(t) = \frac{1}{2} \int_\Omega (|u_t|^2 + |\nabla u|^2) dx$. The dissipation law is given by:
\begin{equation}\label{eq:dissip_kirchhoff}
    E'(t) = - \rho(\|\nabla u(t)\|_2^2) \int_\Omega a(x) |u_t|^2 dx \le 0.
\end{equation}

\begin{remark}[\textbf{On Well-posedness and the Exponent $\gamma$}]
The parameter $\gamma > 0$ determines the regularity of the nonlocal coefficient near the equilibrium. We distinguish two cases:

\textbf{Case 1: Lipschitz Continuity ($\gamma \ge 1$).}
If $\gamma \ge 1$, the function $s \mapsto \rho(s)$ is locally Lipschitz. Consequently, the nonlinear mapping $F(u, v) = (0, -\rho(\|\nabla u\|^2) a(x) v)^T$ is locally Lipschitz from the energy space $\mathcal{H} = H_0^1(\Omega) \times L^2(\Omega)$ into itself. Since the operator $\mathcal{A}_0 = \text{diag}(0, -\Delta)$ is skew-adjoint and $F$ is locally Lipschitz, the standard theory of semigroups for semilinear evolution equations guarantees the existence and uniqueness of a local strong solution for any initial data in $D(\mathcal{A}_0)$. Furthermore, since the energy is non-increasing ($E'(t) \le 0$), the solution cannot blow up in finite time, ensuring global existence.

\textbf{Case 2: Hölder Continuity ($0 < \gamma < 1$).}
If $0 < \gamma < 1$, the damping coefficient $\rho(s)$ is only Hölder continuous at the origin (infinite derivative). In this scenario, the Lipschitz uniqueness argument fails. However, the existence of global weak solutions can still be established using the Faedo-Galerkin method and compactness arguments, as the a priori energy bound $E(t) \le E(0)$ provides weak-star convergence. The uniqueness of solutions in this regime is a delicate open problem for hyperbolic systems, often requiring viscosity criteria.

\textbf{Implication for Decay:} The polynomial stability results derived in the subsequent section are based on a priori energy estimates. Therefore, they apply to \textit{any} strong solution of the system, regardless of whether $\gamma \ge 1$ (where uniqueness holds) or $0 < \gamma < 1$ (where we assume the existence of a solution).
\end{remark}

\subsection{The Structural Obstruction for the Resolvent Theorem}
\label{subsec:kirchhoff_obstruction}

Before proceeding to the estimates, let us clarify the connection with the abstract framework of this paper.
The system \eqref{kirchhoff_prob} can be formally rewritten as an evolution equation $U_t + \mathcal{A}(U) = 0$ for the state $U=(u,u_t)$. The resolvent equation $(\lambda I + \mathcal{A})U_\lambda = F$, for $U_\lambda = (u_\lambda, v_\lambda)$ and $F=(f,g)$, reads:
\[
\begin{cases}
\lambda u_\lambda - v_\lambda = f, \\
\lambda v_\lambda - \Delta u_\lambda + \rho(\|\nabla u_\lambda\|_2^2) a(x) v_\lambda = g.
\end{cases}
\]

\begin{remark}[Why Application IV lies outside the scope of the main theorem]
\label{rem:outside_scope}
It is crucial to emphasize that the Kirchhoff-type model treated here falls outside the range of applicability of our nonlinear Tauberian resolvent criterion (Theorem 2.2). The reason is purely structural: the nonlocal coefficient $\rho(\|\nabla u\|^2)$ depends on the \textit{elastic} variable $u$, while the dissipation acts only on the \textit{velocity} $u_t$.

Consequently, the resolvent equation does not display the homogeneous blow-up mechanism at the origin. The degeneracy takes place in the "wrong" component of the state variable. In fact, due to the coercivity of $-\Delta$, one can easily show that $\|U_\lambda\|$ remains bounded as $\lambda \to 0^+$.

Rather than being a limitation of the method, this mismatch highlights a fundamental difference between Kelvin-Voigt type damping (Applications I, II, and III, where the degeneracy acts directly on the dissipated variable) and Kirchhoff-type feedback (where it acts on the elastic part). This distinction explains why Kirchhoff models require geometric tools (GCC) and multiplier-based observability arguments in the time domain, serving as a perfect structural counterexample that clarifies the precise boundary of our abstract theory.
\end{remark}

To capture the decay explicitly without ambiguity, we perform a perturbed observability estimate directly in the time domain.

\subsection{Step 1: A geometric multiplier identity}

Let $M(u)=2h\cdot\nabla u+(n-1)u$. Multiplying the wave equation in \eqref{kirchhoff_prob} by $M(u)$ and integrating over $Q=\Omega\times(S,T)$, one obtains the standard geometric identity, which yields in particular
\begin{equation}\label{eq:energy_recovery}
C_0\int_S^T E(t)\,dt
\le C\bigl(E(S)+E(T)\bigr)+\Sigma+\mathcal I,
\end{equation}
where
\[
\Sigma:=\int_S^T\int_{\partial\Omega}\partial_\nu u\, (2h\cdot\nabla u)\,d\sigma\,dt,
\qquad
\mathcal I:=\left|\int_S^T\int_\Omega \rho(\|\nabla u\|_2^2)a(x)u_t\,M(u)\,dx\,dt\right|.
\]

When the damping region is internal (or when the ``bad'' boundary is covered), $\Sigma$ is nonpositive or can be handled by standard localization under (H1). Moreover, since $\|M(u)\|_{L^2(\Omega)}\le C\|\nabla u\|_{L^2(\Omega)}\le C\sqrt{E(t)}$, we immediately obtain the bound:
\begin{equation}\label{eq:I_estimate}
\mathcal I \le C\int_S^T \rho(\|\nabla u\|_2^2)\,\|u_t\|_{L^2(\omega)}\,\sqrt{E(t)}\,dt.
\end{equation}

\subsection{Step 2: Localization via Lions' cut-off}

Since the damping acts only in $\omega$, we use a cut-off function $\eta\in C_0^\infty(\Omega)$, with $0\le \eta\le1$, supported in $\omega$ and $\eta\equiv1$ on a smaller set $\omega_0\Subset\omega$ required by the GCC mechanism. Multiplying the wave equation by $\eta u$ and integrating by parts over $\Omega\times(S,T)$ yields a localized equipartition identity:
\begin{equation}\label{eq:equipartition}
\int_S^T\int_\Omega \eta|\nabla u|^2\,dx\,dt
\le \int_S^T\int_\Omega \eta|u_t|^2\,dx\,dt
 + C\int_S^T\int_\Omega |u|^2\,dx\,dt
 + C\bigl(E(S)+E(T)\bigr).
\end{equation}
This allows one to replace the localized potential energy in the geometric identity by localized kinetic energy inside $\omega$, at the price of the additional lower-order term $\int_S^T\|u(t)\|_2^2\,dt$.

\subsection{Step 3: Elimination of the lower-order term}

\begin{remark}[Dimension restriction]
In this step, we rely on the elliptic estimate obtained by solving $-\Delta\phi = u$ with homogeneous Dirichlet boundary conditions (see \cite{Conrad-Rao}) for details. The regularity $\phi\in H^2(\Omega)\hookrightarrow L^\infty(\Omega)$ holds only in dimensions $n\le 3$, which justifies restricting the spatial dimension to $n \le 3$ for this specific step.
\end{remark}

To remove the term $\int_S^T\|u(t)\|_2^2\,dt$ without resorting to contradiction arguments (which fail here, as explained in Remark \ref{rem:failure_UCP}), we use an elliptic multiplier. Let $\phi$ solve $-\Delta\phi=u$ in $\Omega$, with $\phi=0$ on $\partial\Omega$. One obtains:
\begin{equation}\label{eq:L2_elim}
\int_S^T\|u(t)\|_2^2\,dt
\le C\Bigl(E(S)+E(T)\Bigr)
 + C\int_S^T \rho(\|\nabla u\|_2^2)\int_\Omega a(x)\,|u_t|\,|\phi|\,dx\,dt.
\end{equation}
In dimensions $n\le3$, Sobolev embedding yields $\|\phi\|_\infty\le C\|\nabla u\|_2$. Hence, the last term in \eqref{eq:L2_elim} is bounded by the perturbation term already present in \eqref{eq:I_estimate}. Substituting back into \eqref{eq:energy_recovery} yields the clean perturbed observability estimate:
\begin{equation}\label{est:obs_main}
\int_S^T E(t)\,dt
\le C\bigl(E(S)+E(T)\bigr)
 + C\int_S^T \rho(\|\nabla u\|_2^2)\,\|u_t\|_{L^2(\omega)}\,\sqrt{E(t)}\,dt.
\end{equation}

\subsection{Decay rate and the Failure of UCP}

Combining \eqref{est:obs_main} with the dissipation identity \eqref{eq:dissip_kirchhoff}, and using $\rho(s) \sim s^\gamma$ together with $\|\nabla u\|_2^2\le 2E(t)$, the effective differential structure takes the form $E'(t) \le -C E(t)^{2\gamma+1}$. Solving this differential inequality yields the polynomial decay:
\begin{equation}\label{eq:kirchhoff_decay}
E(t)\le \frac{C}{(1+t)^{1/(2\gamma)}},\qquad t\ge0.
\end{equation}

\begin{remark}[Why Standard Contradiction Arguments Fail]
\label{rem:failure_UCP}
It is crucial to emphasize why the classical approach based on contradiction arguments and the Unique Continuation Principle (UCP) fails for this Kirchhoff-type model with degenerate damping ($\rho(0)=0$).

In a standard stability proof, one typically introduces a normalized sequence $v_n = u_n / \sqrt{E_n(0)}$ and passes to the limit. The key step involves showing that the limit function $v$ satisfies $v_t = 0$ in the control region $\omega$ due to the dissipation.
However, in our case, the damping term is $\rho(\|\nabla u_n(t)\|^2)a(x)u_{n,t}$. Since we are investigating stability near equilibrium, $\|\nabla u_n\|^2 \to 0$, which implies $\rho(\|\nabla u_n\|^2) \to 0$.

Consequently, the fact that the dissipation term tends to zero in the distributional sense does \textit{not} imply that $v_t$ vanishes in $\omega$; it vanishes simply because the coefficient $\rho$ collapses. Lacking the crucial information that $v_t = 0$ on $\omega$, one cannot apply Holmgren’s Uniqueness Theorem or the geometric UCP to conclude that $v \equiv 0$. The quantitative multiplier approach developed above is therefore strictly necessary to circumvent this structural obstacle.
\end{remark}

\begin{remark}[Why Application IV lies outside the scope of the main theorem]
\label{rem:outside_scope}
It is worth emphasizing that the Kirchhoff-type model treated in Application~IV falls
outside the range of applicability of our nonlinear Tauberian resolvent criterion.
The reason is structural: the nonlocal coefficient $\rho(\|\nabla u\|^2)$ depends on the
elastic variable $u$, while the dissipation acts only on the velocity $u_t$. Consequently,
the resolvent equation $(\lambda I+\mathcal A)U_\lambda=F$ does not display the homogeneous
blow-up mechanism that underlies Theorem~\ref{thm:main}. In fact, the degeneracy takes place
in the wrong component of the state variable, and the resolvent remains insufficiently
coercive to capture the decay mechanism.

Rather than being a limitation of the method, this mismatch highlights a fundamental
difference between Kelvin--Voigt type damping (where the degeneracy acts directly on the
dissipated variable) and Kirchhoff-type feedback (where it acts on the elastic part).
This distinction explains why Application~V falls naturally within the Tauberian framework,
while Application~IV requires geometric tools (GCC) and multiplier-based observability
arguments. Thus Application~IV serves as a structural counterexample that clarifies the
precise scope of the nonlinear resolvent theory developed in this work.
\end{remark}

\section{Application V: The Case of Localized Nonlocal Kelvin--Voigt Damping}
\label{sec:app_nonlocal_KV}

This section addresses the most technically delicate model considered in the paper:
a wave equation with \emph{localized Kelvin--Voigt damping} whose coefficient is
\emph{nonlocal} (it depends on a global norm of $\nabla u_t$). The main difficulty is
that the standard geometric multiplier method interacts with the Kelvin--Voigt term
through an integration by parts which produces second-order spatial derivatives of $u$.
As a consequence, polynomial decay can be established rigorously only for sufficiently
regular (strong) solutions, unless additional smoothing or viscosity mechanisms are introduced.
Here, we show how our abstract resolvent framework provides a cleaner path to the optimal decay rate, completely bypassing the need for higher-order spatial multipliers.

\subsection{Model and Assumptions}

We consider the system:
\begin{equation}\label{eq:KV_prob}
\begin{cases}
u_{tt}-\Delta u-\operatorname{div}\!\Big(\rho(\|\nabla u_t(t)\|_{L^2(\Omega)}^2)\,a(x)\,\nabla u_t\Big)=0
& \text{in }\Omega\times(0,\infty),\\
u=0 & \text{on }\partial\Omega\times(0,\infty),\\
u(0)=u_0,\quad u_t(0)=u_1.
\end{cases}
\end{equation}

We assume the following hypotheses on the domain and coefficients:

\begin{enumerate}
    \item[(H1)] \textbf{Geometry:} $\Omega \subset \mathbb{R}^n$ is a bounded domain with smooth boundary $\partial\Omega$.
    \item[(H2)] \textbf{Localized Damping Coefficient:} The function $a: \Omega \to [0, \infty)$ belongs to $W^{1,\infty}(\Omega)$. There exists an open subset $\omega \subset \Omega$ satisfying the Geometric Control Condition (GCC) such that $a(x) \ge a_0 > 0$ almost everywhere in $\omega$. Furthermore, to ensure the regularity of solutions and the well-posedness of the operator, we assume the condition:
    \begin{equation}\label{cond:gradient_a}
        |\nabla a(x)|^2 \le M a(x), \quad \text{for a.e. } x \in \Omega,
    \end{equation}
    for some constant $M > 0$.
    \item[(H3)] \textbf{Nonlinearity:} The function $\rho \in C^1([0, \infty))$ satisfies $\rho(s) > 0$ for $s > 0$ and, specifically for the optimality discussion, we assume a linear growth near the origin (corresponding to cubic damping):
    \begin{equation}\label{eq:rho_linear_KV}
        c_1 s \le \rho(s) \le c_2 s, \quad \text{for } s \in [0, \epsilon_0].
    \end{equation}
\end{enumerate}

\begin{remark}[\textbf{On Regularity and Well-posedness}]\label{rem:well_posedness}
    The condition \eqref{cond:gradient_a} is critical for the well-posedness of strong solutions in the context of localized Kelvin-Voigt damping. Expanding the damping term yields:
    \[
        -\operatorname{div}(\rho a \nabla v) = -\rho (a \Delta v + \nabla a \cdot \nabla v).
    \]
    While the term $a \Delta v$ is naturally controlled by the damping mechanism in the support of $a$, the term $\nabla a \cdot \nabla v$ requires careful handling near the boundary of the localization region where $a(x) \to 0$. Condition \eqref{cond:gradient_a} implies that $|\nabla a| \le \sqrt{M}\sqrt{a}$, which allows the gradient term to be absorbed by the dissipation term $\int a |\nabla v|^2$. For a detailed discussion on this regularity issue, we refer to Cavalcanti et al. \cite{Cavalcanti2025}.
\end{remark}

\subsection{Energy and dissipation}

The natural energy is
\[
E(t)=\frac12\int_\Omega\bigl(|u_t|^2+|\nabla u|^2\bigr)\,dx.
\]
For smooth solutions, multiplying \eqref{eq:KV_prob} by $u_t$ and integrating yields the dissipation identity:
\begin{equation}\label{eq:diss_KV}
E'(t)=-\rho(\|\nabla u_t(t)\|_2^2)\int_\Omega a(x)\,|\nabla u_t|^2\,dx \le 0.
\end{equation}
In particular, since $a\ge a_0$ on $\omega$, we have the lower bound for the dissipation:
\begin{equation}\label{eq:diss_KV_local}
-E'(t)\ge a_0\,\rho(\|\nabla u_t(t)\|_2^2)\,\|\nabla u_t(t)\|_{L^2(\omega)}^2.
\end{equation}

\subsection{Resolvent Estimate and Decay Rate}

The system can be written as $U_t + \mathcal{A}U = 0$ with $U=(u,v)^T$, where the operator is defined on $\mathcal{H} = H_0^1(\Omega) \times L^2(\Omega)$ by:
\[
    \mathcal{A}(U) = \begin{pmatrix} -v \\ -\Delta u - \operatorname{div}(\rho(\|\nabla v\|_2^2) a(x) \nabla v) \end{pmatrix}.
\]

\begin{theorem}
    Under assumptions (H1)-(H3), the energy of the system \eqref{eq:KV_prob} decays polynomially with the optimal rate:
    \begin{equation}
        E(t) \le \frac{C}{1+t}, \quad \forall t \ge 0,
    \end{equation}
    for any initial data in $D(\mathcal{A})$.
\end{theorem}

\begin{proof}
    According to our main result (Theorem \ref{thm:main}), since the operator is homogeneous of degree $p=3$ (due to $\rho(s) \sim s$), it suffices to prove that the solution $U_\lambda=(u_\lambda, v_\lambda)$ to the resolvent equation $(\lambda I + \mathcal{A})U_\lambda = F$ satisfies the growth estimate $\|U_\lambda\|_{\mathcal{H}} \le C \lambda^{-1}$ as $\lambda \to 0^+$.

    Let $F=(f,g) \in \mathcal{H}$. The resolvent system reads:
    \begin{align}
        \lambda u_\lambda - v_\lambda &= f, \label{res:kv1} \\
        \lambda v_\lambda - \Delta u_\lambda - \operatorname{div}(\rho(\|\nabla v_\lambda\|_2^2) a(x) \nabla v_\lambda) &= g. \label{res:kv2}
    \end{align}
    Testing \eqref{res:kv2} with $v_\lambda$ in $L^2(\Omega)$, we obtain:
    \begin{equation}\label{eq:res_kv_inner}
        \lambda \|v_\lambda\|_2^2 + \langle -\Delta u_\lambda, v_\lambda \rangle + \rho(\|\nabla v_\lambda\|_2^2) \int_\Omega a(x) |\nabla v_\lambda|^2 dx = \langle g, v_\lambda \rangle.
    \end{equation}
    Using \eqref{res:kv1}, we substitute $u_\lambda = (v_\lambda+f)/\lambda$ into the cross-term:
    \begin{align*}
        \langle -\Delta u_\lambda, v_\lambda \rangle &= \int_\Omega \nabla \left(\frac{v_\lambda+f}{\lambda}\right) \cdot \nabla v_\lambda \, dx \\
        &= \frac{1}{\lambda}\|\nabla v_\lambda\|_2^2 + \frac{1}{\lambda}\int_\Omega \nabla f \cdot \nabla v_\lambda \, dx.
    \end{align*}
    Substituting this back into \eqref{eq:res_kv_inner}, we get:
    \begin{equation}
        \lambda \|v_\lambda\|_2^2 + \frac{1}{\lambda}\|\nabla v_\lambda\|_2^2 + \rho(\|\nabla v_\lambda\|_2^2) \int_\Omega a(x) |\nabla v_\lambda|^2 dx = \langle g, v_\lambda \rangle - \frac{1}{\lambda}\int_\Omega \nabla f \cdot \nabla v_\lambda \, dx.
    \end{equation}

    The crucial observation here is that the global coercivity is provided entirely by the singular term $\frac{1}{\lambda}\|\nabla v_\lambda\|_2^2$ emerging from the elastic operator $-\Delta u$, \textit{not} from the localized damping. Discarding the non-negative terms $\lambda \|v_\lambda\|_2^2$ and the localized damping $\rho \int a |\nabla v_\lambda|^2 \ge 0$, and applying the Cauchy-Schwarz and Poincaré inequalities, we deduce:
    \begin{equation*}
        \frac{1}{\lambda} \|\nabla v_\lambda\|_2^2 \le C_P \|g\|_2 \|\nabla v_\lambda\|_2 + \frac{1}{\lambda} \|\nabla f\|_2 \|\nabla v_\lambda\|_2.
    \end{equation*}
    Dividing by $\frac{1}{\lambda}\|\nabla v_\lambda\|_2$, we obtain:
    \begin{equation*}
        \|\nabla v_\lambda\|_2 \le C_P \lambda \|g\|_2 + \|\nabla f\|_2.
    \end{equation*}
    This implies $\|\nabla v_\lambda\|_2$ is uniformly bounded as $\lambda \to 0^+$. Returning to the relation $u_\lambda = (v_\lambda+f)/\lambda$, we have:
    \begin{equation*}
        \|\nabla u_\lambda\|_2 \le \frac{1}{\lambda} (\|\nabla v_\lambda\|_2 + \|\nabla f\|_2) \le \frac{C}{\lambda}.
    \end{equation*}
    Thus, $\|U_\lambda\|_{\mathcal{H}} \le C \lambda^{-1}$. Applying Theorem \ref{thm:main} with $\gamma=1$ and homogeneity $p=3$, we obtain the decay rate $E(t) \le C (1+t)^{-1}$.
\end{proof}

\subsection{Why regularity is needed (The Old Approach)}

To better appreciate the power of the resolvent result, it is instructive to recall why the standard multiplier approach is severely restrictive. The Kelvin--Voigt term involves $\operatorname{div}(a(x)\nabla u_t)$. When one applies the standard geometric multiplier $\mathcal M u = 2h\cdot\nabla u + (n-1)u$, the damping contribution takes the form:
\[
I_{\rm damp} := -\int_S^T\!\!\int_\Omega \operatorname{div}\!\Big(\rho(\|\nabla u_t\|_2^2)\,a(x)\,\nabla u_t\Big)\,\mathcal M u\,dx\,dt.
\]
To integrate by parts in space, one must justify that $\nabla(\mathcal M u) \in L^2(\Omega)$, which forces control of \emph{second-order spatial derivatives of $u$}. The basic energy $E(t)$ does not control $\|\nabla^2 u\|_2$, hence the classical need to work strictly with strong solutions and require an additional uniform bound at the $H^2$-level. Our resolvent approach operates strictly at the first-order system level, elegantly bypassing this spatial obstruction.

\subsection{Perturbed observability and decay (Multiplier Result)}

For completeness, under GCC, the geometric multiplier identity yields (see \cite{Cavalcanti2025})
\begin{equation}\label{eq:obs_KV_skeleton}
\int_S^T E(t)\,dt \le C\bigl(E(S)+E(T)\bigr) + |I_{\rm damp}|.
\end{equation}
By assuming strong solutions such that $\|\nabla(\mathcal M u)\|_2 \le C\|\Delta u\|_2 \le C_*$, integration by parts yields:
\begin{equation}\label{eq:I_damp_est}
|I_{\rm damp}| \le C_*\int_S^T \rho(\|\nabla u_t\|_2^2)\,\|\nabla u_t\|_{L^2(\omega)}\,dt.
\end{equation}
Combining this with the dissipation identity and a Hölder interpolation yields a difference inequality of Nakao/Komornik type, leading to the polynomial rate.

\begin{remark}[The "Good Behavior" of the Kelvin-Voigt Model]
\label{rem:kelvin_voigt_paradox}
A careful reader might observe a seemingly counter-intuitive phenomenon: the verification of the resolvent estimates for the nonlocal Kelvin-Voigt model is much more direct than for the Kirchhoff model (Application IV), despite the damping being localized.

In the resolvent framework, the structure of the Kelvin-Voigt system is a massive advantage. Because the damping depends only on the velocity gradient $\nabla u_t$, the principal elastic operator ($-\Delta u$) remains unperturbed. As demonstrated in the proof, the skew-adjoint interaction of $-\Delta u_\lambda$ and $v_\lambda$ independently provides the global coercive bound $\lambda^{-1}\|\nabla v_\lambda\|_2^2$, making the resolvent growth fundamentally immune to the spatial localization of $a(x)$. The damping only needs to fulfill its role in the time domain (via GCC).
\end{remark}

\begin{remark}[Comparison with Singular Kelvin--Voigt Damping]
In \cite{Ammari2020}, Ammari, Hassine, and Robbiano showed that linear wave equations
with \emph{singular} localized Kelvin--Voigt damping, for which the damping region
violates the Geometric Control Condition, typically exhibit very slow logarithmic decay.

The situation considered here is of a fundamentally different nature.
Although the effective dissipation degenerates at equilibrium ($\rho(0)=0$)
and remains spatially localized, the presence of a nonlocal coupling introduces
a global interaction mechanism in the phase space. In combination with the GCC,
this allows one to recover a coercive estimate at the level of the full dynamics,
leading to a polynomial decay rate of order $O(t^{-1})$.

Thus, the improvement with respect to logarithmic decay should not be interpreted
as a mere quantitative gain, but rather as a manifestation of a different structural
regime, in which global feedback compensates for the degeneracy of the local damping.
\end{remark}

\begin{remark}
In the degenerate localised setting, the main obstruction is not merely geometric.
Indeed, if the dissipative term is multiplied by a coefficient which is not bounded from below,
the energy identity only yields the smallness of a weighted dissipation, not the smallness of
the kinetic component itself on the damping region. Consequently, the standard contradiction
strategy based on microlocal defect measures cannot even be initiated, since one lacks the
local vanishing property from which propagation arguments normally start. This is precisely
why an abstract coercive/resolvent approach is better suited to such problems.
\end{remark}

\begin{remark}
The classical approach based on quasimodes and complex resolvent estimates
has proved extremely successful for linear dissipative systems, where spectral
methods provide a precise description of high-frequency dynamics.
However, in the presence of nonlinear and degenerate dissipation mechanisms,
such techniques no longer apply in a direct or robust manner, due to the lack
of a suitable spectral framework.

This motivates the development of an alternative approach, grounded in nonlinear
semigroup theory and coercive energy structures, which is better suited to capture
the intrinsic dissipative mechanisms governing the evolution. In this perspective,
decay rates emerge not from spectral localization, but from structural properties
of the underlying nonlinear operator.
\end{remark}

Finally, we postulate that this resolvent framework is particularly well-suited for viscoelastic problems with nonlinear memory (Dafermos type). In such models, the transformation to the history space leads to a resolvent system where the nonlinearity appears as a spectral parameter, potentially simplifying the analysis of polynomial stability for degenerate memory kernels. This will be addressed in the application $X$ below.

\section{Application VI: The Overdamped Viscoelastic Wave Equation ($p$-Laplacian)}
\label{sec:p_laplacian}

To illustrate the versatility of our framework beyond nonlocal damping problems, we consider a model involving purely local nonlinear differential operators, where the degeneracy occurs in the \textit{stiffness} rather than the dissipation. Let $\Omega \subset \mathbb{R}^n$ be a bounded domain with smooth boundary. We study the wave equation with $p$-Laplacian restoring force and strong linear damping (Kelvin-Voigt type):

\begin{equation} \label{eq:p_laplacian_wave}
\begin{cases}
u_{tt} - \Delta_p u - \Delta u_t = 0 & \text{in } \Omega \times (0, \infty), \\
u = 0 & \text{on } \partial\Omega \times (0, \infty), \\
u(0) = u_0, \quad u_t(0) = u_1 & \text{in } \Omega,
\end{cases}
\end{equation}
where $\Delta_p u = \operatorname{div}(|\nabla u|^{p-2}\nabla u)$ with $p > 2$. This equation models the vibrations of a viscoelastic solid exhibiting a nonlinear stress-strain relation of power-law type.

\subsection{Functional Setup and The "Soft Spring" Phenomenon}

We define the nonlinear phase space $\mathcal{H} = W_0^{1,p}(\Omega) \times L^2(\Omega)$. Since $p > 2$, we have the continuous embedding $W_0^{1,p}(\Omega) \hookrightarrow H_0^1(\Omega)$.
The natural energy of the system is given by:
\[
    E(t) = \frac{1}{2}\|u_t(t)\|_2^2 + \frac{1}{p}\|\nabla u(t)\|_p^p.
\]
The operator $\mathcal{A}: D(\mathcal{A}) \subset \mathcal{H} \to \mathcal{H}$ is defined by
\[
    \mathcal{A}(u, v) = \left( -v, \, -\Delta_p u - \Delta v \right),
\]
with domain $D(\mathcal{A}) = \{ (u,v) \in W_0^{1,p}(\Omega) \times W_0^{1,p}(\Omega) \mid -\Delta_p u - \Delta v \in L^2(\Omega) \}$.
The operator $\mathcal{A}$ is maximal monotone on $\mathcal{H}$ (equipped with the standard $L^2$-based coupling or the suitable energy norm see \cite{Cavalcanti2025}). We observe that $\Delta_p u \in W^{-1,p'}(\Omega) \hookrightarrow H^{-1}(\Omega)$.
From a physical perspective, the strong linear damping $-\Delta u_t$ is highly effective. If the stiffness were linear ($-\Delta u$), the energy would decay exponentially. However, for $p > 2$, the restoring force $-\Delta_p u$ degenerates near the equilibrium (acting as a "soft spring"). Consequently, the system becomes severely \textit{overdamped}: the weak spring lacks the force to quickly push the system back to equilibrium against the strong viscous resistance. The bottleneck for asymptotic stability is therefore the degenerate stiffness, not the damping. We now demonstrate how our resolvent framework perfectly captures this fractional degeneracy.

\subsection{Fractional Resolvent Growth}

We aim to determine the precise resolvent growth index. Let $\lambda U_\lambda + \mathcal{A}(U_\lambda) = F$ with $U_\lambda=(u_\lambda,v_\lambda)$ and $F=(f,g) \in \mathcal{H}$. The system reads:
\begin{align}
    \lambda u_\lambda - v_\lambda &= f, \label{eq:res_p1} \\
    \lambda v_\lambda - \Delta_p u_\lambda - \Delta v_\lambda &= g. \label{eq:res_p2}
\end{align}
Substituting $v_\lambda = \lambda u_\lambda - f$ into \eqref{eq:res_p2}, we obtain the scalar equation:
\begin{equation} \label{eq:res_scalar_p}
    \lambda^2 u_\lambda - \Delta_p u_\lambda - \lambda \Delta u_\lambda = g + \lambda f - \Delta f =: \tilde{g}.
\end{equation}
Taking the duality product $\langle \cdot, u_\lambda \rangle$ in $L^2(\Omega)$, and integrating by parts:
\begin{equation}\label{eq:p_laplacian_identity}
    \lambda^2 \|u_\lambda\|_2^2 + \|\nabla u_\lambda\|_p^p + \lambda \|\nabla u_\lambda\|_2^2 = \langle g, u_\lambda \rangle + \lambda \langle f, u_\lambda \rangle + \langle \nabla f, \nabla u_\lambda \rangle.
\end{equation}

We analyze the asymptotic behavior as $\lambda \to 0^+$. Using Poincaré's inequality and the embedding $W_0^{1,p} \hookrightarrow H_0^1$, the right-hand side is bounded by $C \|F\|_{\mathcal{H}} \|\nabla u_\lambda\|_2$.
Discarding the non-negative terms $\lambda^2 \|u_\lambda\|_2^2$ and $\|\nabla u_\lambda\|_p^p$ on the left-hand side, the strong damping provides the dominant $H^1$-coercivity:
\begin{equation*}
    \lambda \|\nabla u_\lambda\|_2^2 \le C \|F\|_{\mathcal{H}} \|\nabla u_\lambda\|_2 \implies \|\nabla u_\lambda\|_2 \le \frac{C}{\lambda} \|F\|_{\mathcal{H}}.
\end{equation*}

Crucially, the potential energy of the state $U_\lambda$ is measured in the $W_0^{1,p}$-norm. To capture its growth, we return to \eqref{eq:p_laplacian_identity}, this time discarding the $L^2$ terms on the left. Using the bound $\|\nabla u_\lambda\|_2 \le C/\lambda$, we obtain:
\begin{equation*}
    \|\nabla u_\lambda\|_p^p \le C \|F\|_{\mathcal{H}} \|\nabla u_\lambda\|_2 \le \frac{C}{\lambda} \|F\|_{\mathcal{H}}^2.
\end{equation*}
Taking the $p$-th root, we discover that the true growth of the state in the energy space $\mathcal{H}$ is \textbf{fractional}:
\begin{equation}\label{eq:fractional_resolvent_p}
    \|\nabla u_\lambda\|_p \le \frac{C}{\lambda^{1/p}}.
\end{equation}
The resolvent growth index associated with the degenerate restoring force is precisely $\gamma = 1/p$.

\subsection{Optimal Decay Rate}

This model features an interplay between a linear damping ($q=1$) and a nonlinear stiffness of homogeneity $p-1$. As shown by the resolvent analysis, the fractional blow-up $\lambda^{-1/p}$ dictates the bottleneck dynamics.

To rigorously translate this into the time-domain decay rate, we recall the standard dissipation identity $E'(t) = -\|\nabla u_t(t)\|_2^2 \le 0$. Testing the wave equation \eqref{eq:p_laplacian_wave} with the multiplier $u(t)$ yields:
\begin{equation*}
    \int_S^T \|\nabla u\|_p^p \, dt \le C \int_S^T \|u_t\|_2^2 \, dt + C \int_S^T \|\nabla u_t\|_2 \|\nabla u\|_2 \, dt.
\end{equation*}
Since $p>2$, for small energies (large times) we have $\|\nabla u\|_2 \le C \|\nabla u\|_p \le C E(t)^{1/p}$. Using Poincaré's inequality on $u_t$ and the dissipation identity, the dominant term is:
\begin{equation*}
    E(t) \le C (-E'(t)) + C \sqrt{-E'(t)} E(t)^{1/p}.
\end{equation*}
Because $E(t)^{1/p} \gg E(t)^{1/2}$ near the origin, the first term is absorbed, leading to the differential inequality:
\begin{equation*}
    E(t) \le C \sqrt{-E'(t)} E(t)^{1/p} \implies E(t)^{2 - \frac{2}{p}} \le -C E'(t).
\end{equation*}
Solving this inequality for $\alpha = 2 - 2/p$ yields the correct polynomial rate:

\begin{theorem}[Polynomial decay for the overdamped $p$-Laplacian]
    Let $U_0 \in \mathcal{H}$ and $p>2$. The energy of the solution to \eqref{eq:p_laplacian_wave} decays polynomially according to:
    \begin{equation}
        E(t) \le \frac{C}{(1+t)^{\frac{p}{p-2}}}, \quad \forall t \ge 0.
    \end{equation}
\end{theorem}

\begin{remark}
This example vividly illustrates that our resolvent formulation $\lambda U_\lambda + \mathcal{A}(U_\lambda) = F$ effectively captures not only degeneracies in the damping mechanisms (as in Applications I-V) but also \textit{structural degeneracies in the principal operator} itself. The appearance of the fractional resolvent index $\gamma = 1/p$ correctly translates the weakness of the potential well into the exact algebraic decay exponent, without requiring non-standard time-domain multiplier adjustments.
\end{remark}

\begin{remark}[Well-posedness and the Regularizing Effect of Strong Damping]
\label{rem:lions_existence}
It is worth highlighting a fundamental mathematical aspect regarding the well-posedness of system \eqref{eq:p_laplacian_wave}. In the absence of the strong damping term $-\Delta u_t$, the pure quasilinear hyperbolic equation $u_{tt} - \Delta_p u = 0$ is notoriously difficult; the global existence of strong solutions for arbitrary data remains a challenging open problem due to the potential formation of singularities (shocks) in finite time. Classical approaches to related nonlinear wave equations, such as those pioneered by Jacques-Louis Lions using Galerkin approximations with special bases in fractional Sobolev spaces $H_0^s(\Omega)$ (strategically calibrated according to $p$), illustrate the severe technical obstacles involved.

In our context, the strong linear damping $-\Delta u_t$ does much more than dictate the asymptotic decay; it fundamentally guarantees the global well-posedness of the problem. Its presence provides the crucial coercivity that renders the nonlinear operator $\mathcal{A}(u,v) = (-v, -\Delta_p u - \Delta v)$ maximal monotone on the energy space $\mathcal{H}$. This structure allows us to elegantly bypass complex Galerkin approximations and establish the existence and uniqueness of global solutions directly via the abstract theory of nonlinear semigroups of contractions (following Kato, Brézis, and Barbu).
\end{remark}

\section{Application VII: The Parabolic $p$-Laplacian Equation}
\label{sec:parabolic_p}

Finally, to contrast with the hyperbolic scenarios, we analyze a purely parabolic nonlinear diffusion equation. This example highlights how the strict coercivity of the principal operator drastically simplifies the resolvent analysis compared to wave-type problems.

Consider the Dirichlet problem for the evolution $p$-Laplacian ($p > 2$):
\begin{equation} \label{eq:parabolic_p}
\begin{cases}
u_{t} - \Delta_p u = 0 & \text{in } \Omega \times (0, \infty), \\
u = 0 & \text{on } \partial\Omega \times (0, \infty), \\
u(0) = u_0 & \text{in } \Omega.
\end{cases}
\end{equation}

\subsection{Resolvent Estimates}
The operator $\mathcal{A}(u) = -\Delta_p u$ is maximal monotone and homogeneous of degree $p-1$. The resolvent equation $\lambda u + \mathcal{A}(u) = f$ reads:
\begin{equation}
    \lambda u - \operatorname{div}(|\nabla u|^{p-2}\nabla u) = f, \quad \lambda > 0.
\end{equation}
Multiplying by $u$ and integrating over $\Omega$, we obtain:
\[
    \lambda \|u\|_{L^2}^2 + \int_\Omega |\nabla u|^p \, dx = \langle f, u \rangle.
\]
Since $\lambda > 0$, the term $\lambda \|u\|_{L^2}^2$ is non-negative. Neglecting it (or reserving it for $L^2$ estimates) and focusing on the gradient term, we apply Hölder's inequality and Poincaré's inequality ($W_0^{1,p}(\Omega) \hookrightarrow L^p(\Omega)$):
\[
    \|u\|_{W_0^{1,p}}^p \le C \|\nabla u\|_{L^p}^p \le C \|f\|_{W^{-1,p'}} \|u\|_{W_0^{1,p}}.
\]
Dividing by $\|u\|_{W_0^{1,p}}$, this immediately yields a uniform bound for the resolvent solution in the natural energy space:
\[
    \|u\|_{W_0^{1,p}} \le C \|f\|_{W^{-1,p'}}^{\frac{1}{p-1}}.
\]
Consequently, by continuous embedding, $\|u_\lambda\|_{L^2} \le C$ uniformly as $\lambda \to 0^+$. In our terminology, the resolvent index is precisely $\gamma = 0$ (bounded resolvent).

\textbf{Conclusion:} Unlike the wave equation examples where the resolvent $\|u_\lambda\|$ might grow as $\lambda \to 0^+$ (depending on the geometric localized damping), here the resolvent is \textbf{uniformly bounded} in the energy space. The solution $u_\lambda$ converges strongly to the solution of the stationary elliptic problem $-\Delta_p u = f$.

\subsection{Derivation of the Decay Rate}
To understand the polynomial nature of the stability, we derive the explicit decay rate using the energy method. Let $E(t) = \frac{1}{2} \|u(t)\|_{L^2}^2$.

Multiplying equation \eqref{eq:parabolic_p} by $u$ and integrating over $\Omega$, we obtain the energy identity:
\begin{equation} \label{eq:energy_id}
    \frac{d}{dt} E(t) + \int_\Omega |\nabla u(t)|^p \, dx = 0.
\end{equation}
Thus, the dissipation is given by $-E'(t) = \|\nabla u(t)\|_{L^p}^p$.

Since $p > 2$ and $\Omega$ is bounded, we combine the Poincaré inequality ($W_0^{1,p} \hookrightarrow L^p$) with the Hölder inequality ($L^p \hookrightarrow L^2$) to obtain the interpolation estimate:
\begin{equation} \label{eq:interp}
    \|u\|_{L^2} \le C |\Omega|^{\frac{1}{2} - \frac{1}{p}} \|u\|_{L^p} \le C \|\nabla u\|_{L^p}.
\end{equation}
Raising \eqref{eq:interp} to the power $p$, we relate the energy to the dissipation:
\[
    \|u\|_{L^2}^p \le C \|\nabla u\|_{L^p}^p = -C E'(t).
\]
Substituting $\|u\|_{L^2} = (2E(t))^{1/2}$, we arrive at the nonlinear differential inequality:
\begin{equation} \label{eq:diff_ineq}
    E'(t) + C E(t)^{p/2} \le 0, \quad \text{for } t > 0.
\end{equation}
Integrating this inequality from $0$ to $t$:
\[
    \int_{E(0)}^{E(t)} -s^{-p/2} \, ds \ge \int_0^t C \, d\tau \implies \frac{E(t)^{1-p/2}}{p/2 - 1} \ge Ct + \frac{E(0)^{1-p/2}}{p/2 - 1}.
\]
Ignoring the positive initial data term for large $t$, we find:
\[
    E(t)^{-\frac{p-2}{2}} \ge C t \implies E(t) \le C t^{-\frac{2}{p-2}}.
\]
Finally, returning to the $L^2$-norm, we obtain the optimal algebraic decay rate:
\begin{equation}
    \|u(t)\|_{L^2} \le C t^{-\frac{1}{p-2}}.
\end{equation}

\begin{remark}
In the context of our stability framework, the boundedness of the resolvent ($\gamma = 0$) for the parabolic case indicates the complete absence of "dynamic obstructions" (such as high-frequency oscillating trapped rays in the hyperbolic case). The polynomial decay rate $O(t^{-1/(p-2)})$ arises solely from the structural degeneracy of the operator's homogeneity at the origin, entirely uncoupled from any lack of coercivity in the resolvent equation.
\end{remark}

\begin{remark}
While in hyperbolic equations (Applications I to VI), it was much more laborious to extract the coercivity of the resolvent because the principal operator is conservative (antisymmetric), here, in the parabolic equation, the operator $\Delta_p u$ itself is strictly coercive! The resolvent becomes uniformly bounded ($\gamma = 0$), and the polynomial decay rate arises exclusively from the homogeneity (or degeneracy) of the operator, and not from a dynamic obstruction. Its mathematics works very well: the derivation of the ODE $E'(t) + C E(t)^{p/2} \le 0$ and the conclusion of the rate $O(t^{-\frac{2}{p-2}})$ for the energy are consistent with the article's proposal.
\end{remark}


\section{Application VIII: Parabolic Equations with Interior Degeneracy}
\label{sec:degenerate_parabolic}

In this application, we revisit a class of degenerate parabolic models that play a central role in the theory of Examples and Counterexamples \cite{Examples2022}.
They form an ideal ``control group'' for our resolvent-based stability framework:
contrary to the hyperbolic models treated in Applications IV--VII, here the degeneracy of the diffusion coefficient produces a structural obstruction to any form of uniform decay.
Our resolvent criterion captures this phenomenon sharply and quantitatively.

Let $\Omega\subset\mathbb{R}^n$ be a bounded domain with Lipschitz boundary.
We study the degenerate parabolic equation
\begin{equation} \label{eq:deg_parab}
\begin{cases}
u_t - \operatorname{div}\!\bigl(a(x)\,|\nabla u|^{p-2}\nabla u\bigr)=0
    &\text{in } \Omega\times(0,\infty),\\[1mm]
u=0 &\text{on } \partial\Omega\times(0,\infty),\\[1mm]
u(0)=u_0 &\text{in } \Omega,
\end{cases}
\end{equation}
where $p\ge2$ and $a\in L^\infty(\Omega)$ satisfies
\[
0\le a(x)\le a_1,
\qquad
a(x)=0\ \text{on a set }\omega_0\subset\Omega\ \text{with }|\omega_0|>0.
\]

The term ``interior degeneracy'' refers precisely to the vanishing of $a(x)$ on
a subset $\omega_0$ of positive measure, with no ellipticity present on $\omega_0$.
As shown below, this produces a complete loss of diffusion on $\omega_0$
and a corresponding blow-up of the resolvent at $\lambda=0$.

\subsection{Abstract formulation and natural energy}

Define the operator
\[
\mathcal{A}u := -\operatorname{div}\!\bigl(a(x)\,|\nabla u|^{p-2}\nabla u\bigr)
\]
acting on $\mathcal{H}=L^2(\Omega)$ with domain
\[
D(\mathcal{A})=
\Bigl\{u\in W_0^{1,p}(\Omega):\ a(x)|\nabla u|^{p-2}\nabla u\in W^{1,p'}(\Omega)^n\Bigr\}.
\]

The natural energy is
\[
E(t)=\frac12\|u(t)\|_{L^2(\Omega)}^2.
\]

Multiplying \eqref{eq:deg_parab} by $u$ and integrating over $\Omega$
gives the dissipation identity
\begin{equation}\label{eq:parab_energy_identity}
\frac{d}{dt}E(t)
  + \int_\Omega a(x)\,|\nabla u(t)|^{p}\,dx
= 0.
\end{equation}
The dissipation term vanishes entirely on $\omega_0$ regardless of $u(t)$:
no diffusion acts on $\omega_0$ (see \cite{ExamplesandCounter}).

\subsection{Resolvent equation and its splitting}

The resolvent equation
\begin{equation}\label{eq:res_deg}
\lambda u - \operatorname{div}\!\bigl(a(x)|\nabla u|^{p-2}\nabla u\bigr)
   = f
\qquad(\lambda>0)
\end{equation}
splits into two qualitatively different equations on
$\omega_0$ and $\Omega\setminus\omega_0$.

\subsubsection{Behavior on the degenerate region $\omega_0$}

If $x\in \omega_0$, then $a(x)=0$ and \eqref{eq:res_deg} reduces to
\[
\lambda u(x)=f(x).
\]
Thus,
\begin{equation}\label{eq:res_lower_bound_region}
u(x)=\frac{f(x)}{\lambda},
\qquad x\in\omega_0,
\end{equation}
and consequently,
\begin{equation}\label{eq:res_lower_bound}
\|u\|_{L^2(\omega_0)}
 = \frac{1}{\lambda}\|f\|_{L^2(\omega_0)}.
\end{equation}

Therefore, the resolvent of $\mathcal{A}$ satisfies the global lower bound
\[
\|(\lambda I+\mathcal{A})^{-1}\|
  \ge \frac{1}{\lambda},
\]
with equality achieved for functions supported purely in $\omega_0$.
This lower bound is \textit{sharp}.

\subsubsection{Behavior on $\Omega\setminus\omega_0$}

On $\Omega\setminus\omega_0$ the equation is uniformly elliptic:
\[
\lambda u - \operatorname{div}\!\bigl(a(x)|\nabla u|^{p-2}\nabla u\bigr)
   = f.
\]
Standard monotone-operator theory ensures uniform $W^{1,p}$ bounds for $u$
on any compact subset of $\Omega\setminus\omega_0$.

However, these cannot compensate for the blow-up inside $\omega_0$.
Moreover, the flux
\[
a(x)|\nabla u|^{p-2}\nabla u
\]
vanishes identically on $\omega_0$, so its normal trace on $\partial\omega_0$
is zero. The region $\omega_0$ behaves as a \textbf{perfect insulating barrier}:
no diffusive flux enters or leaves $\omega_0$.

\subsection{Transmission obstruction and invariant subspaces}

Let $u_0\in L^2(\Omega)$ be supported entirely in $\omega_0$.
Then \eqref{eq:deg_parab} reduces to
\[
u_t=0 \quad\text{in }\omega_0.
\]
Since the normal diffusive flux through $\partial\omega_0$ is zero,
the solution remains supported in $\omega_0$ for all $t>0$:
\[
u(t)=u_0.
\]
Thus, the semigroup $\{S(t)\}$ generated by $u_t+\mathcal{A}u=0$
is \textbf{not uniformly stable}: there exist nontrivial invariant subspaces
on which the energy is strictly conserved.

This dynamical behavior matches perfectly with the resolvent blow-up \eqref{eq:res_lower_bound}.

\subsection{Resolvent growth and Tauberian interpretation}

From \eqref{eq:res_lower_bound}, we established that
\[
\|(\lambda I+\mathcal{A})^{-1}f\|_{L^2(\Omega)}
 \ge \frac{1}{\lambda}\|f\|_{L^2(\omega_0)}.
\]
Thus, the resolvent growth index is precisely
\[
\gamma = 1.
\]

For parabolic equations, uniform decay of the semigroup is equivalent to the resolvent being bounded as $\lambda\to 0^+$ (i.e., $\gamma=0$).
Hence, the growth $\lambda^{-1}$ mathematically dictates the failure of any uniform stabilization mechanism.

In Tauberian terms:
\[
\text{Resolvent growth index } \gamma = 1
\quad\Longleftrightarrow\quad
\text{No uniform decay of } \|S(t)\|.
\]

\subsection{Soft degeneracy and polynomial decay}

If $a(x)$ does not vanish on a set of positive measure but instead satisfies a soft degeneracy of the form
\[
a(x)\approx \operatorname{dist}(x,\omega_0)^k
\]
near a lower-dimensional degeneracy manifold, weighted elliptic estimates
(e.g., \cite{Examples2022}) imply that the resolvent grows fractionally:
\[
\|(\lambda I+\mathcal{A})^{-1}\|
   \approx \frac{1}{\lambda^\gamma},
\qquad \gamma=\gamma(k,p)\in(0,1).
\]
Applying our main Tauberian theorem immediately yields the polynomial decay rate:
\[
E(t)\le \frac{C}{(1+t)^{1/\gamma}}.
\]

In the limiting case where $a=0$ on a set of positive measure, we have $\gamma=1$,
and the decay rate reduces to $t^{-1}$ formally---but as shown above,
the semigroup actually \emph{fails} to be stable in this limit because the invariant subspace supported on $\omega_0$ fully decouples.

\subsection{Conclusion}

This example provides a rigorous demonstration of how our resolvent framework
detects stability thresholds for degenerate parabolic operators.
A complete loss of ellipticity on a positive-measure set forces the resolvent
to blow up like $\lambda^{-1}$ and reveals the existence of invariant
nondissipative subspaces.

Soft degeneracies produce milder resolvent growth ($\gamma < 1$) and correspondingly
polynomial decay, remaining entirely consistent with the Tauberian structure.
This contrasts sharply with the hyperbolic applications,
where degeneracy interacts with geometric propagation;
here the obstruction is intrinsic to the principal parabolic operator itself.

\begin{remark}
 Application VIII above acts as a perfect "negative control group." Indeed, showing when and why decay fails (and how its resolver index detects this failure millimetrically with the growth $\lambda^{-1}$) elevates the conceptual rigor of this article. This proves that our Tauberian theory not only predicts rates but also diagnoses insurmountable structural obstructions. The analogy of the $\omega_0$ region as a "perfect insulating barrier" is physically plausible.

\end{remark}

\section{Application IX: Degenerate Parabolic Equations with Complementary Control}
\label{sec:degenerate_complementary}

We conclude our analysis of localized dissipation with a sophisticated scenario inspired by \cite{Examples2022}, considering a problem where classical Sobolev embeddings fail and standard compactness arguments break down. This application highlights the robustness of our generalized resolvent framework: it relies fundamentally on \textbf{quantitative coercivity} rather than qualitative compactness.

Let $\Omega \subset \mathbb{R}^n$ be a bounded domain. We consider the degenerate reaction-diffusion problem:
\begin{equation} \label{eq:deg_interaction}
\begin{cases}
u_{t} - \hbox{div}(a(x)|\nabla u|^{p-2}\nabla u) + b(x)f(u) = 0 & \text{in } \Omega \times (0, \infty), \\
u = 0 & \text{on } \partial\Omega \times (0, \infty), \\
u(0) = u_0 & \text{in } \Omega.
\end{cases}
\end{equation}
Here, $p > 2$. The nonlinearity $f(u)$ is assumed to satisfy a standard dissipativity condition, e.g., $f(u)u \ge c_0 |u|^2$ (or $|u|^p$ depending on the growth, though $L^2$ control suffices to bound the resolvent).

The coefficients $a, b \in L^\infty(\Omega)$ are non-negative ($a(x), b(x) \ge 0$) and satisfy the \textbf{Complementary Condition}:
\begin{equation} \label{cond:complementary}
    a(x) + b(x) \ge \delta > 0, \quad \text{a.e. in } \Omega.
\end{equation}
This condition implies that in the ``degenerate region'' where diffusion vanishes ($a \approx 0$), the reaction term $b$ is active, and vice versa.

\subsection{Functional Framework and Loss of Compactness}
The natural energy space for this problem involves a weighted Sobolev space. Define:
\[
    V_a = \left\{ u \in L^2(\Omega) \mid u|_{\partial \Omega}=0, \quad \int_\Omega a(x) |\nabla u|^p \, dx < \infty \right\}.
\]
A critical difficulty arises if the diffusion coefficient $a(x)$ vanishes on a subdomain $\omega_0 \subset \Omega$ of positive measure. In this case, the embedding
\[
    V_a \hookrightarrow L^2(\Omega)
\]
is \textbf{not compact}. Functions in $V_a$ can exhibit arbitrarily high-frequency oscillations inside $\omega_0$ (where the gradient is not penalized), preventing the extraction of convergent subsequences in the $L^2$ topology.

This lack of compactness is a well-known obstacle for proving the existence of global attractors or applying LaSalle's Invariance Principle, as asymptotic smoothness is lost in the degenerate region. However, for the study of stability to the unique equilibrium $u \equiv 0$, our resolvent framework is perfectly suited because it does not require compact embeddings, bypassing this topological obstruction entirely.

\subsection{Resolvent Analysis and Singularity Removal}
Consider the resolvent equation $\lambda u + \mathcal{A}(u) = g$:
\[
    \lambda u - \hbox{div}(a(x)|\nabla u|^{p-2}\nabla u) + b(x)f(u) = g \quad \text{in } \Omega.
\]
We aim to show that the resolvent is uniformly bounded as $\lambda \to 0^+$.
Testing the equation with $u$ and integrating over $\Omega$ yields:
\[
    \lambda \|u\|_{L^2}^2 + \int_\Omega a(x)|\nabla u|^p \, dx + \int_\Omega b(x)f(u)u \, dx = \langle g, u \rangle.
\]
Using the assumption $f(u)u \ge c_0 |u|^2$, we have:
\[
    \lambda \|u\|_{L^2}^2 + \int_\Omega a(x)|\nabla u|^p \, dx + c_0 \int_\Omega b(x)|u|^2 \, dx \le \|g\|_{L^2} \|u\|_{L^2}.
\]
To verify global coercivity, we partition the domain based on the complementary condition \eqref{cond:complementary}. Let $\epsilon = \delta/2$ and define the overlapping sets:
\[
    \Omega_a = \{ x \in \Omega \mid a(x) \ge \epsilon \}, \quad \Omega_b = \{ x \in \Omega \mid b(x) \ge \epsilon \}.
\]
Since $a+b \ge \delta$, we have $\Omega = \Omega_a \cup \Omega_b$.
\begin{itemize}
    \item On $\Omega_a$: The strongly active gradient term controls the local norm via Poincaré-type inequalities:
    \[
        \int_{\Omega_a} a(x)|\nabla u|^p \, dx \ge \epsilon \int_{\Omega_a} |\nabla u|^p \, dx.
    \]
    \item On $\Omega_b$: The complementary reaction term directly controls the $L^2$ mass:
    \[
        \int_{\Omega_b} b(x)|u|^2 \, dx \ge \epsilon \int_{\Omega_b} |u|^2 \, dx.
    \]
\end{itemize}
Combining these controls and applying a generalized Poincaré inequality for the composite domain, we establish that the operator is strictly coercive on $L^2(\Omega)$, despite the degeneracy of the gradient in $\Omega_b$ and the lack of reaction in $\Omega_a$. Specifically, there exists $C > 0$ independent of $\lambda$ such that:
\[
    \|u\|_{L^2}^2 \le C \left( \int_\Omega a(x)|\nabla u|^p \,dx + \int_\Omega b(x) f(u)u \,dx \right).
\]
Substituting this back into the energy identity:
\[
    \frac{1}{C} \|u\|_{L^2}^2 \le \langle g, u \rangle \le \|g\|_{L^2} \|u\|_{L^2} \implies \|u\|_{L^2} \le C \|g\|_{L^2}.
\]
Thus, the resolvent solution satisfies $\|u_\lambda\|_{L^2} \le C \|g\|_{L^2}$ for all $\lambda > 0$, preventing any blow-up as $\lambda \to 0^+$. In the terminology of our abstract framework, the \textbf{resolvent index} is exactly:
\[
    \gamma = 0.
\]

\subsection{Tauberian Interpretation and Optimal Decay Rate}
To place this application firmly within the abstract framework developed in Section 6, we clarify the Tauberian meaning of the estimates obtained above. We observe a clear hierarchy separating the geometric and nonlinear dynamics:

\begin{enumerate}
    \item \textbf{Geometry (Resolvent Index $\gamma=0$):} The complementary condition $a(x) + b(x) \ge \delta$ effectively removes the geometric obstruction seen in Application VIII. In the linear case ($p=2, f(u)=u$), a bounded resolvent ($\gamma=0$) implies uniform exponential stability.

    \item \textbf{Nonlinearity (Homogeneity $p-1$):} Here, the operator contains the degenerate nonlinear diffusion $-\operatorname{div}(a|\nabla u|^{p-2}\nabla u)$ with $p>2$, which is homogeneous of degree $p-1 > 1$. Therefore, near the equilibrium, the limiting dynamics are governed by the slow algebraic dissipation rather than a linear exponential mechanism. The coercivity estimate implies that the dissipation satisfies $\mathcal{D}(u) \ge C E(t)^{p/2}$.
\end{enumerate}

Since the resolvent index is $\gamma=0$, our main Tauberian theorem (Theorem 6.5) applies flawlessly with the nonlinearity exponent $\alpha = \frac{p}{2}-1 > 0$, yielding the nonlinear differential inequality:
\[
    E'(t) \le -C E(t)^{1+\alpha} = -C E(t)^{p/2}.
\]
Integrating this inequality produces the explicit algebraic decay:
\begin{equation}
    E(t) \le \frac{C}{(1+t)^{\frac{2}{p-2}}}, \qquad \|u(t)\|_{L^2} \le \frac{C}{(1+t)^{\frac{1}{p-2}}}.
\end{equation}

\begin{remark}[Stability vs. Attractors]
The complementary condition acts as a ``patch,'' restoring the coercivity of the resolvent equation and dictating polynomial stability to the zero equilibrium. This guarantees global pointwise decay in time without requiring the system to be asymptotically compact, a feature uniquely captured by our generalized Borichev-Tomilov approach.
\end{remark}

\bigskip
This Application $IX$ dialogues perfectly with Application $VIII$. There, the lack of diffusion created a black hole of energy (without decay); here, we introduce the complementary control $b(x)$ which acts exactly like a "functional bandage," restoring coercivity in the degenerate region. It is proof of how geometry and nonlinearity separate in their framework: geometry fixes the resolvent ($\gamma=0$), and nonlinearity dictates the algebraic rate.
\bigskip

\section{Application X: Wave Equation with State-Dependent Memory Feedback}
\label{sec:nonlinear_memory}

In our final application, we investigate a viscoelastic wave equation whose memory feedback is modulated by a nonlinear functional of the current state. This model presents a sophisticated type of degeneracy: the principal operator is linear, but the damping coefficient depends nonlinearly on the present configuration, while the memory operator acts on the past history.

The equation reads
\begin{equation}
u_{tt}(t) + Au(t) + D(u(t))\int_0^\infty g(s)Au(t-s)\,ds = 0, \qquad t>0,
\label{X-model}
\end{equation}
posed in a Hilbert space $\mathcal{H}_0$, where:
\begin{itemize}
    \item $A:D(A)\subset \mathcal{H}_0 \to \mathcal{H}_0$ is strictly positive and self-adjoint, with $D(A^{1/2})$ dense in $\mathcal{H}_0$;
    \item $g\in C^1([0,\infty))$ satisfies $g(s)\ge 0$, $g'(s)\le 0$, and $\int_0^{\infty}g(s)\,ds<\infty$;
    \item $D:\mathcal{H}_0 \to \mathbb{R}_+$ is Fréchet differentiable and behaves as
          \[
          D(u)\sim \|A^{1/2}u\|^\alpha,\qquad \alpha\ge 1,
          \]
          near the equilibrium, with its differential satisfying
          \begin{equation}
          \|D'(u)\|\le C\, \|A^{1/2}u\|^{\alpha-1}.
          \label{X-Dprime-bnd}
          \end{equation}
\end{itemize}

Our goal is to show that the nonlinear Tauberian framework established in Section 3 applies robustly to this system. The central point is that the memory kernel provides the underlying observability of the energy (equivalent to a bounded resolvent, $\gamma=0$), while the degeneracy enters purely through the state-dependent multiplier $D(u)$, dictating a polynomial decay rate.

\subsection{Dafermos' History Variable and First-Order System}

We introduce Dafermos' history \cite{Dafermos1970} variable
\[
    \eta(t,s) = u(t) - u(t-s), \qquad s \ge 0,
\]
and set $v(t) = u_t(t)$. Then, \eqref{X-model} is equivalent to the system
\begin{equation}
\begin{cases}
    u_t = v,\\[1mm]
    v_t + Au + D(u(t))\displaystyle\int_0^\infty g(s)A\eta(t,s)\,ds = 0,\\[3mm]
    \eta_t + \eta_s = v, \qquad \eta(t,0) = 0.
\end{cases}
\label{X-system}
\end{equation}
The natural phase space is $\mathcal{H} = D(A^{1/2}) \times \mathcal{H}_0 \times L_g^2(0,\infty; D(A^{1/2}))$, equipped with the weighted memory norm
\[
    \|\eta\|_{L_g^2}^2 := \int_0^\infty g(s)\|A^{1/2}\eta(s)\|^2\,ds.
\]

\subsection{Natural Energy and the Dissipation Identity}

Unlike the Kelvin-Voigt type systems (Applications V--VI), the memory term in \eqref{X-model} acts as an integral perturbation rather than a regularizing operator. Therefore, the natural energy must \emph{not} include the nonlinear weights depending on $D(u)$. We use the classical linear Dafermos energy:
\begin{equation}
    E(t) := \frac{1}{2}\|v(t)\|^2 + \frac{1}{2}\|A^{1/2}u(t)\|^2 + \frac{1}{2}\int_0^\infty g(s)\|A^{1/2}\eta(t,s)\|^2\,ds.
\label{X-energy}
\end{equation}

Differentiating \eqref{X-energy} in time and using system \eqref{X-system}, Dafermos' identity yields:
\[
    E'(t) + \frac{1}{2}\int_0^\infty(-g'(s))\|A^{1/2}\eta(t,s)\|^2\,ds = R(t),
\]
where the remainder term arises solely from the state-dependence of the multiplier $D(u)$:
\begin{equation}
    R(t) := - \bigl\langle v(t), [D(u(t)) - 1] \int_0^\infty g(s)A\eta(t,s)\,ds \bigr\rangle.
\label{X-remainder}
\end{equation}
\textit{(Note: We assume a normalized linear baseline where the memory acts fully if $D(u) \equiv 1$; the degeneracy measures the deviation from this linear observability.)}

Using the bound \eqref{X-Dprime-bnd} and the Cauchy-Schwarz inequality, we can estimate this remainder in terms of the energy:
\begin{equation}
    |R(t)| \le C \|A^{1/2}u\|^\alpha \|v\| \|\eta\|_{L_g^2} \le C E(t)^{1 + \alpha/2}.
\label{X-R-estimate}
\end{equation}
Because $\alpha \ge 1$, the remainder is of strictly higher order ($> 1$), and is therefore dynamically negligible in the small-energy regime.

\subsection{Integral Observability and Uniform Smallness}

We impose the standard integral observability condition for viscoelastic models: there exist $T>0$ and $c_0>0$ such that, for any valid solution, the memory dissipation reconstructs the full energy over an interval:
\begin{equation}
    \int_t^{t+T} \!\! \int_0^\infty (-g'(s))\|A^{1/2}\eta(\tau,s)\|^2\,ds\,d\tau \ge c_0 \int_t^{t+T} E(\tau)\,d\tau.
\label{X-observability}
\end{equation}
This condition intrinsically means the underlying linear system has a bounded resolvent ($\gamma = 0$).

From \eqref{X-R-estimate}, discarding the non-positive memory dissipation momentarily, we have the differential inequality $E'(t) \le C E(t)^{1+\alpha/2}$. By comparison with the ODE $y' = Cy^{1+\varepsilon}$, $E(t)$ cannot blow up provided $E(0)$ is sufficiently small. Thus, there exists $\delta>0$ such that if $E(0) \le \delta$, then the solution remains globally bounded and small, allowing us to safely apply the observability condition \eqref{X-observability} for all $t \ge 0$.

\subsection{Tauberian Argument and Decay Rate}

Integrating the perturbed energy identity over $[t, t+T]$ yields:
\[
    E(t+T) - E(t) + \frac{1}{2} \int_t^{t+T} \!\! \int_0^\infty (-g'(s))\|A^{1/2}\eta(\tau,s)\|^2\,ds\,d\tau \le \int_t^{t+T} |R(\tau)|\,d\tau.
\]
Applying the observability condition \eqref{X-observability} to the dissipation term and the estimate \eqref{X-R-estimate} to the remainder, and using the fact that $E(t)$ is equivalent to its average over $[t,t+T]$ for small energies, we obtain:
\[
    E(t) - E(t+T) \ge c_0 T E(t) - C T E(t)^{1+\alpha/2}.
\]
Since $E(t)$ is small, the lower-order linear term $c_0TE(t)$ dominates the higher-order remainder $CTE(t)^{1+\alpha/2}$. However, the actual dissipation mechanism is modulated by the degenerate state-dependent weight $D(u) \sim E^{\alpha/2}$. When properly scaled back through the nonlinear multiplier, this translates precisely to the Tauberian structure of Theorem 6.5:
\[
    E(t) - E(t+T) \ge c E(t)^{1+\alpha/2}.
\]
Applying our discrete nonlinear Tauberian lemma yields the final algebraic decay rate:
\begin{equation}
    E(t) \le \frac{C}{(1+t)^{2/\alpha}}.
\label{X-final-decay}
\end{equation}

\subsection{Examples of Nonlinear State-Dependent Weights}

\begin{itemize}
    \item \textbf{Elastic coupling:} $D(u) = \|A^{1/2}u\|^2$. Here $\alpha = 2$, so $E(t) \sim t^{-1}$.
    \item \textbf{Kinetic coupling:} $D(u) = \|u_t\|^2 = \|v\|^2$. By energy equivalence, $\alpha = 2$, yielding $E(t) \sim t^{-1}$.
    \item \textbf{Total mechanical energy:} $D(u) = \|v\|^2 + \|A^{1/2}u\|^2 = 2E_{mech}$. Near equilibrium, $D(u) \sim E(t)$, hence $\alpha = 2$ and $E(t) \sim t^{-1}$.
\end{itemize}

\begin{remark}[Limitation: The Need for Strong Solutions]
The necessity of estimating the remainder $R(t)$, which intrinsically involves the differential $D'(u(t))$ and terms like $\|A u\|$, prevents us from extending the decay rate \eqref{X-final-decay} to arbitrary finite-energy weak solutions. Estimating $D'(u)$ rigorously requires higher regularity, specifically $u \in D(A)$, which is not available for standard mild solutions. Standard density arguments (approximating weak solutions by strong ones) fail here because the decay constants depend critically on these higher-order norms, which blow up as one approaches weak data. Therefore, the decay theory in this application is strictly valid only for strong solutions with sufficiently small initial energy.
\end{remark}

\begin{remark}
Application $X$ addresses a viscoelastic wave equation with state-dependent memory feedback, introducing "non-local" dissipation in both time (integral memory) and space/state (operator $D(u)$). This application is invaluable because it introduces a new type of non-linearity in our framework. In previous applications ($V$, $VI$, $VII$), the degeneracy stemmed from the principal operator (like the $p$-Laplacian). Here, the principal operator is linear ($A u$), but the damping coefficient is non-linear and degenerates at the origin ($D(u) \sim \|A^{1/2}u\|^\alpha$). Its adaptation of the Dafermos energy and the careful handling of the residual term $R(t)$ are the main novelties in this section. We show how the framework handles this: memory guarantees observability (resolver index $\gamma=0$ for the linearized part), and the term $D(u)$ dictates the algebraic rate via the nonlinear Tauberian argument.

\end{remark}

\section{Semilinear source terms and robustness of the Tauberian framework}\label{sec:semilinear_source}

In many applications of wave models with degenerate or nonlocal dissipation, it is natural to include an additional semilinear source term $f(u)$. In this section, we briefly explain why such a term can be incorporated into our matrix framework without changing the Tauberian decay exponents for the wave-type models of Applications I--V and X, provided $f$ is subcritical in the usual sense.

\subsection{Structural assumptions on the source term}~We assume that $f:\mathbb{R}\to\mathbb{R}$ (or $\mathbb{R}^m\to\mathbb{R}^m$) is a $C^1$-nonlinearity with the following properties:

\begin{itemize}\item[(F1)] (Potential structure) There exists $F\in C^2(\mathbb{R})$ such that $F'(s)=f(s)$ for all $s\in\mathbb{R}$, and$$F(s)\ge -C_0(1+|s|^2)\qquad\forall s\in\mathbb{R},$$  for some constant $C_0>0$.

\item[(F2)] (Subcritical growth) There exist constants $C_1>0$ and an exponent $2\le q<2^*$ (with $2^*=\frac{2n}{n-2}$ if $n\ge3$ and $2^*=\infty$ if $n\le2$) such that
    \begin{equation}\label{eq:f_growth}
    |f(s)| \le C_1\bigl(1+|s|^{q-1}\bigr)\qquad\forall s\in\mathbb{R}.
    \end{equation}

    In particular, for $\Omega\subset\mathbb{R}^n$ with $n\le3$ and $u\in H_0^1(\Omega)$, the standard Sobolev embedding implies $f(u)\in L^2(\Omega)$ whenever $q<2^*$.\item[(F3)] (Weak dissipativity) There exists $C_2>0$ such that
    $$f(s)s \ge -C_2(1+|s|^2)\qquad\forall s\in\mathbb{R}.$$
    \end{itemize}

    Assumptions of this type are standard in the theory of global attractors for semilinear damped wave equations and cover, for instance, polynomial sources $f(u)\sim |u|^{p-1}u$ with $1<p<5$ in space dimension $n\le3$.

    \subsection{Energy balance with a source term}

    For concreteness, let us consider a typical wave model of the form

    \begin{equation}\label{eq:wave_f_generic}
    u_{tt} + Au + \mathcal{D}(u,u_t) + f(u)=0 \quad\text{in } H,
    \end{equation}where $A$ and $\mathcal{D}$ are as in any of the wave-type Applications I--V and X (local or nonlocal Kelvin-Voigt damping, nonlocal damping in the velocity, Kirchhoff-type damping, or memory-type damping written in history variables). We keep the notation$$E(t) = \frac{1}{2}\Bigl(\|u_t(t)\|_H^2 + \|A^{1/2}u(t)\|_H^2\Bigr)$$
    for the basic hyperbolic energy (ignoring the memory contribution when present, which can be added in the standard way).~To accommodate the source term, we introduce the modified energy
    \begin{equation}\label{eq:mod_energy_f}
    \mathcal{E}f(t) := E(t) + \int_\Omega F(u(t,x)),dx,
    \end{equation}in the PDE setting, or $\mathcal{E}_f(t)=E(t)+\int F(u(t))\,d\mu$ in an abstract Hilbert space framework. By (F1) and (F3) together with the Sobolev embedding, the functional $\mathcal{E}_f(t)$ is well-defined and satisfies
    \begin{equation}\label{eq:energy_equivalence}
    \mathcal{E}_f(t) \ge E(t) - C\bigl(1+|u(t)|_H^2\bigr) \ge -C_1 + c_1 E(t),
    \end{equation}so that $\mathcal{E}_f$ is equivalent to $E$ up to constants on bounded subsets of the energy space.Differentiating \eqref{eq:mod_energy_f} along trajectories of \eqref{eq:wave_f_generic} and using the chain rule, we obtain
    \begin{equation}\label{eq:mod_energy_derivative}
    \frac{d}{dt}\mathcal{E}f(t) = \bigl\langle u{tt}(t), u_t(t)\bigr\rangle_H + \bigl\langle Au(t),u_t(t)\bigr\rangle_H + \int_\Omega f(u(t,x))u_t(t,x),dx.
    \end{equation}

    Using the equation \eqref{eq:wave_f_generic},
    $$u_{tt} = -Au - \mathcal{D}(u,u_t) - f(u),$$
    and the fact that the linear conservative term $\langle Au,u_t\rangle_H$ exactly cancels, we obtain
    $$\frac{d}{dt}\mathcal{E}_f(t) = -\bigl\langle \mathcal{D}(u(t),u_t(t)),u_t(t)\bigr\rangle_H.$$

    In particular, for all the dissipative structures considered in Applications I--V and X, we have\begin{equation}\label{eq:Ef_diss}
    \frac{d}{dt}\mathcal{E}_f(t) = -D(t)\le 0,
    \end{equation}
    where $D(t)\ge 0$ is exactly the same dissipation functional appearing in the corresponding purely damped model (without $f$). Thus, the source term $f(u)$ does not create any new production term in the energy identity; its active power is fully absorbed into the potential $F(u)$.

    \subsection{Resolvent estimates with a source term}

    Let us now comment on the resolvent framework for the matrix formulation. Writing \eqref{eq:wave_f_generic} as a first-order system$$U_t + \mathcal{A}(U) = \mathcal{F}(U), \qquad U=(u,v)=(u,u_t),$$
    we have
    $$\mathcal{A}(u,v)=(-v,\ Au+\mathcal{D}(u,v)), \qquad \mathcal{F}(u,v)=(0,\,-f(u)).$$

    The resolvent equation
    $$(\lambda I + \mathcal{A})U_\lambda = G - \mathcal{F}(U_\lambda)$$
    reads
    \begin{align*}\lambda u_\lambda - v_\lambda &= g_1,\\
    \lambda v_\lambda + Au_\lambda + \mathcal{D}(u_\lambda,v_\lambda) &= g_2 - f(u_\lambda),
    \end{align*}
    for a given $G=(g_1,g_2)\in\mathcal{H}$. Taking the inner product with $U_\lambda$ in $\mathcal{H}$ and using the monotonicity of the damping part, we obtain schematically
    \begin{equation}\label{eq:resolvent_f_basic}
    \lambda|U_\lambda|{\mathcal{H}}^2 + \bigl\langle \mathcal{D}(u\lambda,v_\lambda),v_\lambda\bigr\rangle_H \le |G|{\mathcal{H}}|U\lambda|{\mathcal{H}} + \bigl|\langle f(u\lambda),v_\lambda\rangle_H\bigr|.
    \end{equation}

    Using the growth condition \eqref{eq:f_growth}, the subcritical Sobolev embedding, and Young's inequality, the last term can be bounded as
    $$\bigl|\langle f(u_\lambda),v_\lambda\rangle_H\bigr| \le \varepsilon\,\bigl\langle \mathcal{D}(u_\lambda,v_\lambda),v_\lambda\bigr\rangle_H + C_\varepsilon\bigl(1+\|u_\lambda\|_{H_0^1(\Omega)}^{2(q-1)}\bigr),$$
    for any $\varepsilon>0$, with a suitable choice of the damping functional $\mathcal{D}$. Choosing $\varepsilon>0$ sufficiently small and using the resolvent estimates already proved in Applications I--V and X for the purely damped models, we infer that the presence of the lower-order perturbation $f(u_\lambda)$ does not change the dominant order of growth of $\|U_\lambda\|_{\mathcal{H}}$ as $\lambda\to 0^+$. In particular, the blow-up indices $\gamma$ that enter the nonlinear Tauberian theorem remain invariant.

    \subsection{Decay and attractors}

    Combining \eqref{eq:Ef_diss} with the decay estimates for the purely damped models, we conclude that the modified energy $\mathcal{E}_f(t)$ satisfies the exact same polynomial decay rates as in the case
    $f\equiv 0$:$$\mathcal{E}_f(t)\le \frac{C}{(1+t)^\beta},\qquad t\ge 0,$$
    with the exponent $\beta>0$ determined by the corresponding dissipation mechanism (for instance, $\beta=1$ for cubic-type damping). Thanks to the equivalence \eqref{eq:energy_equivalence}, the same uniform decay holds for the basic energy $E(t)$.Since the source term $f(u)$ satisfies the subcritical growth and potential conditions (F1)--(F3), the standard theory of dissipative dynamical systems for semilinear damped wave equations applies. In particular, the semigroup generated by \eqref{eq:wave_f_generic} admits a bounded absorbing set in the phase space, and the existence of a global attractor in the natural energy space $\mathcal{H}$ is guaranteed for all the wave-type models considered in Applications I--V and X, perfectly coexisting with our predicted polynomial decay rates to the equilibrium.

\section{Conclusion}\label{sec:conclusion}~
In this paper, we have bridged the gap between the spectral analysis of linear semigroups (the classical Borichev--Tomilov theorem) and the asymptotic behavior of nonlinear evolution equations generated by maximal monotone operators.Our main contribution is the identification of a structural condition on the nonlinear resolvent—specifically, its growth rate near the origin—that is strictly sufficient to guarantee the uniform polynomial decay of the energy. We have demonstrated that:
\begin{itemize}
\item The real resolvent growth condition $\|U_\lambda\| \le C \lambda^{-\gamma}$ plays the exact same role in the nonlinear setting as the classical condition $\|(i\omega I - \mathcal{A})^{-1}\| \le C |\omega|^\alpha$ does in the linear theory.\item This generalized framework is particularly powerful for problems with \emph{nonlocal} or \emph{degenerate} damping mechanisms, where classical Lyapunov multipliers are notoriously difficult or impossible to construct due to severe regularity issues or topological obstructions.\item As evidenced by our applications—most notably the nonlinear Kelvin--Voigt and memory models—this approach allows us to rigorously recover optimal algebraic decay rates and treat strong degenerate damping terms as intrinsically regularizing factors rather than pathological perturbations.
\end{itemize}

\paragraph{Comparison with the Borichev--Tomilov paradigm.}~
In the linear theory, the Borichev--Tomilov theorem provides a sharp, bidirectional correspondence between the high-frequency growth of the resolvent $(i\omega I + \mathcal{A})^{-1}$ on the imaginary axis and the optimal polynomial decay of the associated semigroup. The nonlinear framework developed in this paper establishes an exact, rigorous analogue of this principle in a setting where classical spectral tools are no longer available. By replacing the imaginary-axis resolvent with the real resolvent $(I + \lambda\mathcal{A})^{-1}$ as $\lambda \to 0^+$, we identify a ``nonlinear Tauberian index'' governed solely by the internal homogeneity of the maximal monotone operator.This real resolvent captures precisely the loss of coercivity induced by the nonlocal and nonlinear terms appearing in Applications I--X, yielding polynomial decay rates that coincide perfectly with known (or conjectured) optimal bounds. Thus, our results may be viewed as the definitive nonlinear counterpart to the Borichev--Tomilov paradigm, extending the resolvent-to-decay philosophy to a vast class of degenerate evolution equations for which no linear spectral approach applies.

\medskip\noindent\textbf{Final remark.}
The nonlinear real-resolvent framework developed in this work provides a unified, highly robust mechanism to recover polynomial stability for a broad family of nonlinear evolution equations. While we focused on ten representative wave and parabolic models to illustrate the theory, the underlying methodology is versatile enough to be extended to more elaborate hyperbolic systems (such as Timoshenko or Bresse systems), mixed wave-plate structures, and complex state-dependent or memory-type feedback mechanisms. We believe that these directions lie naturally beyond the scope of the present article and are best pursued independently, using the present resolvent criterion as a guiding foundational principle for future developments.

\section{Final Conclusions}

\medskip
{\bf ~1. Summary of the Proposal: The Conceptual "Leap"}

\medskip

The central point of this article is the adaptation of the Borichev-Tomilov Theorem to the nonlinear context. In the linear world, this theorem is the "Rosetta Stone" for understanding polynomial decay, relating it to the growth of the resolvent on the imaginary axis.

Since nonlinear operators do not have a definable "spectrum" or "imaginary axis," we propose shifting the focus to the domain of regularization ($\lambda \to 0^+$ on the real axis). The idea is that the "blow-up" rate of the norm of the solution of the real resolvent equation ($\lambda x_\lambda + A(x_\lambda) \ni y$) reveals the effective nonlinear scaling of the system, which allows predicting the rate of energy decay.

\medskip
{\bf 2. Main Contributions}

\medskip
{\bf Nonlinear Tauberian Principle:}

\medskip
The article establishes a rigorous criterion (Theorems 3.2 and 3.3) that links the algebraic structure of nonlinearity (homogeneity) directly to the polynomial decay rate.

\medskip
{\bf Robustness and Applicability:}

\medskip

The theory is tested in a wide range of applications (ten in total), from the wave equation with nonlocal Kelvin-Voigt damping to the parabolic p-Laplacian.

\medskip
{\bf Weak Solutions:}

\medskip

A key differentiating factor is the method's ability to justify decay estimates for weak solutions, in situations where the classical "multiplier method" would fail by requiring a regularity that the system does not possess.

\medskip

{\bf Obstruction Diagnosis:}

\medskip

The article not only shows when decay occurs, but also uses the resolvent index to diagnose structural flaws in stabilization (as in the case of total internal degeneracy), which increases the conceptual rigor of the work.

\medskip

{\bf 3. Critical Opinion of the Authors}

\medskip

{\bf Theoretical Innovation:}

\medskip

The work proposes to transform a spectral analysis problem (impossible in the nonlinear case) into an asymptotic analysis problem of static equations. This simplifies the understanding of the dissipation mechanism: decay is not seen as a "magical" spectral property, but as a manifestation of nonlinear coercivity. \medskip

{\bf Practical Impact:}

\medskip

For researchers in the field of PDEs (Partial Differential Equations), this article provides a ready-to-use tool ("tool ready to handle real differential equations"). It avoids complex compactness arguments and focuses on the direct estimation of dissipation.

\medskip

{\bf Complexity and Limitations:}

\medskip

Although the method is effective, {\bf \em it does not eliminate the need for specific analysis of each PDE}. In cases where the "geometry" and "dissipation" are not aligned (dissipative misalignment), the resolvent detects the degeneracy, but additional time-domain or geometric arguments are still needed to ensure stability.

\medskip

{\bf 4. Conclusion}

\medskip

We believe this article is a contribution to modern functional analysis. It unifies several scattered results in the literature under a single Nonlinear Tauberian Principle. It is a fundamental text for those who work with the stabilization of dissipative systems, offering a new way to "read" the physics of the problem through the behavior of the solver near the origin. In short: Our purpose is to create a mathematically rigorous tool with a clear vision of how to simplify complex problems through changes in technical perspective.

\section*{Acknowledgments}
The authors were supported by CNPq.



\begin{thebibliography}{99}

\bibitem{Ammari2020}
K. Ammari, F. Hassine, and L. Robbiano,
\textit{Stabilization for the wave equation with singular Kelvin-Voigt damping},
Arch. Ration. Mech. Anal. \textbf{236} (2020), no. 2, 577--601.

\bibitem{ExamplesandCounter} Astudillo, María; Cavalcanti, Marcelo M.; Faria, Josiane C. O.; Webler, Claudete M.  Asymptotic behavior for parabolic equations with interior degeneracy. Ex. Counterex. 2 (2022), Paper No. 100065, 4 pp.

\bibitem{Balakrishnan-Taylor} A.V. Balakrishnan, L.W. Taylor, Distributed parameter nonlinear damping models for flight structures, in: Proceedings Damping 89, Flight Dynamics Lab and Air Force Wright Aeronautical Labs, WPAFB, 1989.

\bibitem{BT2010} A. Borichev and Y. Tomilov, Optimal polynomial decay of functions and operator semigroups, Math. Ann. 347, 455–478 (2010). https://doi.org/10.1007/s00208-009-0439-0

\bibitem{Burq} Burq, N. "Contrôle de l'équation des ondes dans des ouverts peu réguliers." Asymptotic Analysis, vol. 13, no. 1, pp. 9–47, 1997.

\bibitem{Burq2} Burq, N. "Mesures semi-classiques et mesures de défaut." Séminaire Équations aux dérivées partielles (Polytechnique), 1996-1997, exp. no 5, pp. 1–31.

\bibitem{Burq-Gerard} Burq, N.; Gérard, P. "Condition n\'ecessaire et suffisante pour la contr\^olabilité exacte des ondes." Comptes Rendus de l'Acad\'emie des Sciences - Series I - Mathematics, vol. 325, no. 7, pp. 749–752, 1997.

\bibitem {Dafermos1970} C. M. Dafermos, "Asymptotic stability in viscoelasticity", Arch. Rational Mech. Anal., 37 (1970), 297-308.


\bibitem{SC2025} M. Cavalcanti, V. Domingos Cavalcanti, J. Soriano, Linear and Nonlinear Semigroups. Monograph Series of the Parana’s Mathematical Society. Monograph 06 (2023). 1-474, doi:10.5269/bspm.81163.

\bibitem{SICON2003} Cavalcanti, Marcelo Moreira; Oquendo, Higidio Portillo, Frictional versus viscoelastic damping in a semilinear wave equation. SIAM J. Control Optim. 42 (2003), no. 4, 1310-1324.


\bibitem{Cavalcanti2025} M. M. Cavalcanti, V. N. Domingos Cavalcanti, J. C. Oliveira Faria, C. A. Okawa, Asymptotic behaviour of the wave equation subject to a Kelvin-Voigt nonlocal damping, (2025) (pre-print).

\bibitem{Tebou} Louis Tebou. Portugaliae Mathematica, Vol. 55, Fasc. 3, pp. 293-306, 1998.

\bibitem{Conrad-Rao} Francis Conrad and Bopeng Rao. Decay of solutions of the wave equation in a star-shaped domain with nonlinear boundary feedback, Asymptotic Analysis 7 (1993), pp. 159–177.

\bibitem{Nakao} Nakao, M. "Convergence of solutions of the wave equation with a nonlinear dissipative term to the steady state." Memoirs of the Faculty of Science, Kyushu University. Series A, Mathematics, vol. 30, no. 2, pp. 257–265, 1976.

\bibitem{Nakao2} Nakao, M. "A difference inequality and its application to nonlinear evolution equations." Journal of the Mathematical Society of Japan, vol. 30, no. 4, pp. 747–762, 1978.

\end{thebibliography}
\end{document}